\documentclass{amsart}
\usepackage{amsmath,amssymb}
\usepackage{graphicx}
\usepackage{url}
\usepackage{xypic}\xyoption{tips}
\usepackage{comment}
\excludecomment{comment}
\usepackage{float}
\usepackage{tikz}
\usetikzlibrary{patterns}
\usetikzlibrary{decorations.text}
\usetikzlibrary{positioning}
\usetikzlibrary{matrix}
\usetikzlibrary{arrows}
\usetikzlibrary{decorations.pathreplacing,decorations.markings}
\usetikzlibrary{shapes.misc,shapes.multipart}
\usepackage{tkz-euclide}
\usetikzlibrary{intersections}
\usepackage{tikz-cd}
\usetikzlibrary{tikzmark,quotes,arrows}
\tikzset{commutative diagrams/.cd,mysymbol/.style={start anchor=center, end anchor=center,draw=none}}

\usepackage{array}
\newcolumntype{L}[1]{>{\raggedright\let\newline\\\arraybackslash\hspace{0pt}}m{#1}}
\newcolumntype{C}[1]{>{\centering\let\newline\\\arraybackslash\hspace{0pt}}m{#1}}
\newcolumntype{R}[1]{>{\raggedleft\let\newline\\\arraybackslash\hspace{0pt}}m{#1}}

\newcommand\dhxrightarrow[2][]{%
	\mathrel{\ooalign{$\xrightarrow[#1\mkern4mu]{#2\mkern4mu}$\cr%
			\hidewidth$\rightarrow\mkern4mu$}}
}

\newtheorem{theorem}{Theorem}
\newtheorem{lemma}[theorem]{Lemma}
\newtheorem{proposition}[theorem]{Proposition}
\newtheorem{corollary}[theorem]{Corollary}
\newtheorem{definition}[theorem]{Definition}
\newtheorem{example}[theorem]{Example}

\newcommand{\isom}{\cong}

\newcommand{\FF}{\mathbb{F}}

\newcommand{\Aut}{\mathrm{Aut}}
\newcommand{\Gal}{\mathrm{Gal}}
\newcommand{\Galk}{\Gal(\bar{k}/k)}
\newcommand{\End}{\mathrm{End}}
\newcommand{\Hom}{\mathrm{Hom}}

\newcommand{\Cl}{{\mathcal{C}\!\ell}}

\renewcommand{\SS}{\mathrm{SS}}
\newcommand{\Tr}{\mathrm{Tr}}
\newcommand{\Nr}{\mathrm{Nr}}
\newcommand{\disc}{\mathrm{disc}}

\newcommand{\PP}{\mathbb{P}}
\newcommand{\QQ}{\mathbb{Q}}

\newcommand{\ZZ}{\mathbb{Z}}

\newcommand{\DD}{\Delta}
\renewcommand{\aa}{\mathfrak{a}}
\newcommand{\pp}{\mathfrak{p}}
\newcommand{\qq}{\mathfrak{q}}
\newcommand{\qqi}{\mathfrak{q}_{(i)}}

\newcommand{\cG}{\mathcal{G}}

\newcommand{\OO}{\mathcal{O}}

\newcommand{\XX}{\mathcal{X}}

\newcommand{\ignore}[1]{}

\def\sqt{1.732050807568877}   

\def\centerarc[#1](#2)(#3:#4:#5)
{ \draw[#1] ($(#2)+({#5*cos(#3)},{#5*sin(#3)})$) arc (#3:#4:#5); }

\title{Orienting supersingular isogeny graphs}

\author{Leonardo Col{\`o} and David Kohel}
\address{Institut de Math\'ematiques de Marseille\\
Aix Marseille Universit\'e}

\begin{document}

\begin{abstract}
We introduce a category of $\OO$-oriented supersingular elliptic curves and
derive properties of the associated oriented and nonoriented $\ell$-isogeny
supersingular isogeny graphs.  As an application we introduce an oriented
supersingular isogeny Diffie-Hellman protocol (OSIDH), analogous to the
supersingular isogeny Diffie-Hellman (SIDH) protocol and generalizing the
commutative supersingular isogeny Diffie-Hellman (CSIDH) protocol.
\end{abstract}


\maketitle

\section{Introduction}

In this paper we introduce a category of supersingular elliptic curves oriented by an
imaginary quadratic order $\OO$, and derive properties of the associated
oriented and non-oriented supersingular $\ell$-isogeny graphs.
This permits one to derive a faithful group action on a subset of oriented supersingular
curves, equipped with a forgetful map to the set of non-oriented supersingular curves.
As an application we introduce an oriented supersingular isogeny Diffie-Hellman
protocol (OSIDH), analogous to the supersingular isogeny Diffie-Hellman (SIDH)
of De Feo and Jao~\cite{JDF2011} and generalizing the commutative supersingular isogeny
Diffie-Hellman (CSIDH) of Castryck, Lange, Martindale, Panny and Renes \cite{CLMPR2018}, the
latter based on the idea of group actions on sets by Couveignes~\cite{C2006} and
Rostovtsev-Stolbunov~\cite{RS2006}.
Renewed interest in these isogeny-based protocols is motivated by their presumed
resistance to quantum attacks, and this work both enlarges the class of
isogeny-based protocols and provides a framework for their security analysis.

We study some theoretical and practical aspects of the endomorphism ring of a supersingular
elliptic curve and their connection with isogeny graphs. The central idea is to use an
embedding of a quadratic imaginary order into the endomorphism ring of a supersingular
elliptic curve, a maximal order in a quaternion algebra, to introduce an orientation on
the curve.
This extra piece of information permits one to impose compatible actions of the class
groups of the suborders of this quadratic order on the descending isogeny chains and
therefore on the isogeny volcano of oriented curves.

We observe that the starting vertex of the chain can be chosen to have a special
orientation (by an order of class number one) and that computations can be performed
using modular polynomials. This motivates us to introduce a Diffie-Hellman key
exchange protocol that avoids limitations imposed by earlier constructions.

The idea of SIDH is to fix a large prime number $p$ of the form
$p={\ell_A^{e_A}\ell_B^{e_B}f\pm1}$ for a small cofactor
$f$ and to let the two parties Alice and Bob take random walks
(i.e., isogenies chains) of length $e_A$ (or $e_B$) in the
$\ell_A$-isogeny graph (or the $\ell_B$-isogeny graph, respectively)
on the set of supersingular $j$-invariants defined over $\FF_{p^2}$.
In order to have the two key spaces of similar size
$\ell_A^{e_A}\approx\ell_B^{e_B}$, we need to take
$\ell_A^{e_A}\approx\ell_B^{e_B}\approx \sqrt{p}$.
Since the total number of supersingular $j$-invariants
is around $p/12$, this implies that, for each party,
the space of choices for the secret key is limited to $1/\sqrt{p}$
of the whole set of supersingular $j$-invariants over $\FF_{p^2}$.
In other words, in choosing their secrets, Alice and Bob can go
only ``halfway'' around the graph from the starting vertex $j_0$.

Recently, Castryck, Lange, Martindale, Panny and Renes proposed
another key exchange protocol based on supersingular isogeny graphs
over the prime field $\FF_p$.
We fix a prime of the form $p=4\ell_1\cdot\ldots\cdot\ell_t-1$ and
an elliptic curve $E/\FF_p$ defined by the equation $E:y^2=x^3+ax^2+x$.
The peculiarity of CSIDH is that it works with curves defined
over $\FF_p$ and restricts the endomorphism rings of such curves
to the commutative subring consisting of $\FF_p$-rational endomorphisms.
Starting from this setup, the scheme is an adaptation of the
Couveignes and Rostovtsev-Stolbunov idea. Observe that the choice
of looking at curves defined over $\FF_p$, instead of $\FF_{p^2}$,
limits the key spaces for Alice and Bob to $\#\Cl(\ZZ[\sqrt{-p}])$
supersingular points. For a given $p$, this is the same order of
magnitude, $O(\sqrt{p}\log(p))$, as for SIDH, but the class group
is transitive on this subset.

In this paper we want to describe a new cryptographic protocol,
the OSIDH, defined over an arbitrarily large subset of oriented
supersingular elliptic curves over $\FF_{p^2}$, which combines
features of SIDH and CSIDH, and permits one to cover an
arbitrary proportion of all isomorphism classes of supersingular
elliptic curves.

A feature shared by SIDH and CSIDH is that the isogenies are
constructed as quotients of rational torsion subgroups: the
secret path of length $e_A$ in the $\ell_A$-isogeny graph
corresponds to a secret cyclic subgroup $\langle A\rangle \subseteq
E\left[\ell^{e_A}\right]$ where $A$ is a rational $\ell_A^{e_A}$-torsion point on $E$.
The need for rational points imposes limits on the choice of the
prime $p$ and, thus, of the finite field we work on.
In contrast OSIDH relies on constructions that can be carried out
only with the use of modular polynomials hence avoiding conditions
on the rational torsion subgroup.

In summary, an orientation provides a class group action on lifts
of an arbitrarily large subset of supersingular points. Exploiting
an effective subring $\OO$ of the full endomorphism ring we obtain
an effective action by the class group of this subring on the
isogeny volcano ({\it whirlpool}). This approach generalizes the
class group action of CSIDH where supersingular elliptic curves
are oriented by the commutative subring $\ZZ\left[\pi\right]$
generated by Frobenius $\pi = \sqrt{-p}$.
To avoid subexponential (or polynomial) time reductions, in the
OSIDH protocol, as detailed in Section~\ref{section:OSIDH}, the
orientation and associated class group action is hidden in the
intermediate data exchanged by Alice and Bob.  This gives a
protocol for which the best known attacks at present are fully
exponential.


\section{Orientations, isogeny chains, and ladders}

In this section, we recall the definition of an isogeny graph and introduce
the notion of orienting supersingular elliptic curves and their isogenies by
an imaginary quadratic field $K$ and its orders $\OO$.  Finally, we describe
how to impose a structure on an isogeny graph by means of isogeny chains and
how to carry out an effective class group action, by means of ladders.

\subsection*{Isogeny graphs}

Given an elliptic curve $E$ over a field $k$, and a finite set of primes
$S$, we can associate an {\it isogeny graph} $\Gamma = \Gamma_S(E)$,
whose vertices are elliptic curves $\bar{k}$-isogenous to $E$, with
fixed vertex $E$, and whose directed edges are isogenies of degree $\ell \in S$.
The vertices are defined up to $\bar{k}$-isomorphism, and the edges
from a given vertex are defined up to a $\bar{k}$-isomorphism of the
codomain.
If $S = \{\ell\}$, then we call $\Gamma$ an $\ell$-isogeny graph, which
we write as $\Gamma_\ell(E)$.

An $\ell$-isogeny graph $\Gamma$ is equiped with an action of $\cG = \Galk$,
with the vertex $[E]$ a fixed point, as follows. We have
$$
E[\ell] = \{ P \in E(\bar{k}) \;|\; \ell P = O \} \cong (\ZZ/\ell\ZZ)^2.
$$
The set of cyclic subgroups is in bijection with $\PP(E[\ell]) \isom \PP^1(\ZZ/\ell\ZZ)$,
which in turn is in bijection with the set of $\ell$-isogenies from $E$.
The $\cG$-action on $E[\ell]$ induces an action by $\cG$ on the $\ell+1$ cyclic subgroups.
This action extends to paths without backtracking of length~$n$, via the
action on the cyclic subgroups $G$ of order $\ell^n$ in
$$
E[\ell^n] = \{ P \in E(\bar{k}) \;|\; \ell^n P = O \} \cong (\ZZ/\ell^n\ZZ)^2.
$$
which are in bijection with $\PP(E[\ell^n]) \isom \PP^1(\ZZ/\ell^n\ZZ)$.
This determines a compatible Galois action on vertices $[E/G]$ and edges
$\varphi: E/G_i \to E/G_{i+1}$ where $G_i \subset G_{i+1}$ is of index~$\ell$.
The action on infinite paths from $E$ is thus determined by the
Galois action on the projective Tate module $\PP(T_\ell(E)) \isom \PP^1(\ZZ_\ell)$.
In the same way we define the $\cG$-action on $\Gamma_S(E)$ derived from
the $\cG$-set structure of $\PP(T_S(E))$, where
$$
T_S(E) = \prod_{\ell \in S} T_\ell(E).
$$
The choice of base curve $E$ determines a Galois action on $\Gamma$,
conjugate to the Galois action induced by a twist of $E$.

Thus an $\ell$-isogeny graph is $(\ell+1)$-regular for outgoing edges.
The existence of curves of $j$-invariant $0$ or $12^3$ with additional
automorphisms in the graph implies a reduced number of incoming edges
at these vertices.
We define an undirected graph $\overline{\Gamma}_\ell(E)$ by identifying an
isogeny $\varphi: E_0 \to E_1$ with its dual $\hat\varphi : E_1 \to E_0$,
and if $\Aut(E_0) \ne \{\pm1\}$ or $\Aut(E_1) \ne \{\pm1\}$ the orbits
$$
\Aut(E_1)\varphi\Aut(E_0) \mbox{ and } \Aut(E_0)\hat\varphi\Aut(E_1)
$$
are identified, which gives a non-bijective correspondence between edges
and dual edges.

\begin{lemma}
Let $E$ be an elliptic curve over $k$ with endomorphism ring $\OO$,
and for a prime $\ell \ne \mathrm{char}(k)$ let $\overline{\Gamma}_\ell(E)$
be its undirected $\ell$-isogeny graph.
\begin{enumerate}
\setlength\itemsep{0mm}
\item
\label{ZZ}
If $\OO = \ZZ$, then each component of $\overline{\Gamma}_\ell(E)$ is an infinite tree.
\item
\label{CM}
If $\OO$ is an order in a CM field $K$, then each component $\overline\Gamma$ of
$\overline{\Gamma}_\ell(E)$ is infinite and either
\begin{itemize}
\setlength\itemsep{0mm}
\item
the prime $\ell$ is split in $K$ and $\overline\Gamma$ has a unique cycle, or
\item
the prime $\ell$ is ramified or inert in $K$ and $\overline\Gamma$ is a tree.
\end{itemize}
\item
\label{QM}
If $\OO$ is an order in a quaternion algebra, then $\overline\Gamma_\ell(E)$ is finite and connected.
\end{enumerate}
\end{lemma}

If $E$ is defined over a number field, then case~\eqref{ZZ} is the generic
case and in the CM case~\eqref{CM}, every curve admits an embedding of an
order of $K$ in its endomorphism ring, and the Galois action is determined
by CM theory (see Shimura~\cite{Shimura}).
If $E$ is defined over a finite field, then only case~\eqref{CM} (ordinary)
or case~\eqref{QM} (supersingular) can hold.  The ordinary case gives rise
to an $\ell$-isogeny graph in bijection with the CM graph with CM field
$K = \QQ(\pi)$, where $\pi$ is the Frobenius endomorphism.
In the supersingular case we have more precisely that there are
$$
\frac{(p-1)}{12} + \frac{1}{3}\left(1 - \left(\!\frac{-3\ }{p}\!\right)\right) + \frac{1}{4}\left(1 - \left(\!\frac{-4\ }{p}\!\right)\right)
$$
vertices.
In the next section we introduce the notion of a $K$-orientation by an
imaginary quadratic field $K$, which allows us to canonically lift the
finite supersingular graph to an infinite oriented CM graph.

\subsection*{Orientations}

Suppose now that $E$ is a supersingular elliptic curve over a finite field $k$
of characteristic $p$, and denote by $\End(E)$ the full endomorphism ring.
We assume moreover that $k$ contains $\FF_{p^2}$ and $E$ is in an isogeny class
such that $\End_k(E) = \End(E)$.

We denote by $\End^0(E)$ the $\QQ$-algebra $\End(E) \otimes_\ZZ \QQ$.
In particular, $\End^0(E)$ is the unique quaternion algebra over $\QQ$
ramified at $p$ and $\infty$.

Let $K$ be a quadratic imaginary field of discriminant $\DD_K$ with maximal
order $\OO_K$.  Then there exists an embedding $\iota : K \to \End^0(E)$
if and only if $p$ is inert or ramified in $\OO_K$, and there exists an
order $\OO \subseteq \OO_K$ such that $\iota(\OO) = \iota(K) \cap \End(E)$.

\begin{definition}
A {\it $K$-orientation} on a supersingular elliptic curve $E/k$ is a
homomorphism $\iota:K\hookrightarrow\End^0(E)$.
An {\it $\OO$-orientation} on $E$ is a $K$-orientation such that
the image of the restriction of $\iota$ to $\OO$ is contained in
$\End(E)$. We write $\End((E,\iota))$ for the order $\End(E)
\cap \iota(K)$ in $\iota(K)$.
An $\OO$-orientation is {\it primitive} if $\iota$ induces an
isomorphism of $\OO$ with $\End((E,\iota))$.
\end{definition}

Let $\phi:E \rightarrow F$ be an isogeny of degree $\ell$.
A $K$-orientation $\iota: K\hookrightarrow\End^0(E)$ determines
a $K$-orientation $\phi_*(\iota):K \hookrightarrow \End^0(F)$
on $F$, defined by
$$
\phi_*(\iota)(\alpha) = \frac{1}{\ell}\,\phi\circ\iota(\alpha)\circ\hat\phi.
$$
Conversely, given $K$-oriented elliptic curves $(E,\iota_E)$ and
$(F,\iota_F)$ we say that an isogeny $\phi : E \rightarrow F$ is
{\it $K$-oriented} if $\phi_*(\iota_E) = \iota_F$, i.e.~if the orientation
on $F$ is induced by $\phi$.
The restriction to $K$-oriented isogenies determines a category of
$K$-oriented elliptic curves, hence of $K$-oriented isomorphism classes,
and a subcategory of $\OO$-oriented elliptic curves.

If $E$ admits a primitive $\OO$-orientation by an order $\OO$ in $K$,
$\phi: E \rightarrow F$ is an isogeny then $F$ admits an induced
primitive $\OO'$-orientation for an order $\OO'$ satisfying
$$
\ZZ + \ell\OO \subseteq \OO' \mbox{ and } \ZZ + \ell\OO' \subseteq \OO.
$$
We say that an isogeny $\phi: E \rightarrow F$ is an $\OO$-oriented isogeny
if $\OO = \OO'$.

If $\ell$ is prime, as direct analogue of Proposition~4.2.23 of~\cite{Kohel1996},
one of the following holds:
\begin{itemize}
\item
$\OO = \OO'$ and we say that $\phi$ is {\it horizontal},
\item
$\OO\subset\OO'$ with  index $\ell$ and we say that $\phi$ is {\it ascending},
\item
$\OO'\subset \OO$ with index $\ell$ and we say that $\phi$ is {\it descending}.
\end{itemize}
Moreover if the discriminant of $\OO$ is $\DD$, then there
are exactly
$
\ell - \left(\frac{\DD}{\ell}\right)
$
descending isogenies. If $\OO$ is maximal at $\ell$, then there are
$
\left(\frac{\DD}{\ell}\right)+1
$
horizontal isogenies, and if $\OO$ is non-maximal at $\ell$, then
there is exactly one ascending $\ell$-isogeny and no horizontal
isogenies.

For an oriented class $(E,\iota)$ with endomorphism ring
$\OO = \End((E,\iota))$, we define $(E,\iota)$ to be at the
{\it surface} (or depth~$0$) if $\OO$ is $\ell$-maximal,
and to be at {\it depth}~$n$ if the valuation at $\ell$
of $[\OO_K:\OO]$ is $n$.
In the next section we introduce $\ell$-isogeny chains linking
oriented curves at the surface to oriented curves at depth~$n$.

The oriented graph $\Gamma_S(E,\iota)$ is the graph whose
vertices are $K$-oriented isomorphism classes, with fixed base
vertex $(E,\iota)$, and whose edges are $K$-oriented $\ell$-isogenies
for $\ell$ in $S$.

\subsection*{Isogeny chains and ladders}

Let $E_0/k$ be a fixed supersingular elliptic curve, equipped with
an $\OO$-orientation, and let $\ell \ne p$ be a prime.

\begin{definition}
We define an {\it $\ell$-isogeny chain} of length $n$ from $E_0$
to $E$ to be a sequence of isogenies of degree $\ell$:
\[
E_0\xrightarrow{~~~\phi_0~~~} 
E_1\xrightarrow{~~~\phi_1~~~}
E_2\xrightarrow{~~~\phi_2~~~}
\ldots
\xrightarrow{~~~\phi_{n-1}~~~} E_n = E.
\]
We say that the $\ell$-isogeny chain is {\it without backtracking}
if $\ker(\phi_{i+1} \circ \phi_{i}) \ne E_{i}[\ell]$ for each
$i = 0,\dots,n-1$, and say that the isogeny chain is {\it descending}
(or {\it ascending}, or {\it horizontal}) if each $\phi_{i}$ is
descending (or ascending, or horizontal, respectively).
\end{definition}

\noindent{\bf Remark.}
Since the dual isogeny of $\phi_i$, up to isomorphism, is the only
isogeny $\phi_{i+1}$ satisfying $\ker(\phi_{i+1} \circ \phi_{i})
= E_{i}[\ell]$, an isogeny chain is without backtracking if and
only if the composition of two consecutive isogenies is cyclic.
Moreover, we can extend this characterization 
in terms of cyclicity to the entire $\ell$-isogeny chain.

\begin{lemma}
The composition of the isogenies in an $\ell$-isogeny chain is cyclic
if and only if the $\ell$-isogeny chain is without backtracking.
\end{lemma}

\noindent{\bf Remark.}
If an isogeny $\phi$ is descending, then the unique ascending isogeny
from $\phi(E)$, up to isomorphism, is the dual isogeny $\hat{\phi}$,
satisfying $\hat{\phi} \phi = [\ell]$.
As an immediate consequence, a descending $\ell$-isogeny chain is
automatically without backtracking, and an $\ell$-isogeny chain without
backtracking is descending if and only if $\phi_0$ is descending.
\vspace{2mm} 

Suppose that $(E_i,\phi_i)$ is an $\ell$-isogeny chain, with $E_0$ equipped
with an $\OO_K$-orientation $\iota_0: \OO_K \rightarrow \End(E_0)$. For each $i$, let $\iota_i : K \rightarrow \End^0(E_i)$ be the induced
$K$-orientation on $E_i$; we note $\OO_i = \End(E_i) \cap \iota_i(K)$ with $\OO_0=\OO_K$ and $\DD_i=\mathrm{discr}(\OO_i)$  with $\DD_0=\DD_K$.

In particular, if $(E_i,\phi_i)$ is a descending $\ell$-chain, then $\iota_i$
induces an isomorphism
$$
\iota_i : \ZZ + \ell^i \OO_K \longrightarrow \OO_i.
$$

Let $q$ be a prime different from $p$ and $\ell$ that splits in $\OO_K$,
let $\qq$ be a fixed prime over $q$.
For each $i$ we set $\qqi = \iota_i(\qq) \cap \OO_i$, and define
$$
C_i=E_i[\qqi] = \{ P \in E_i[q] \;|\; \psi(P) = 0 \mbox{ for all } \psi \in \qqi \}.
$$
We define $F_i = E_i/C_i$, and let $\psi_i: E_i \rightarrow F_i$, an isogeny
of degree $q$. By construction, it follows that $\phi_i(C_i) = C_{i+1}$ for
all $i = 0,\dots,n-1$.
In particular, if $(E_i,\phi_i)$ is a descending $\ell$-ladder, then $\iota_i$
induces an isomorphism
$$
\iota_i : \ZZ + \ell^i \OO_K \longrightarrow \OO_i.
$$
The isogeny $\psi_0: E_0 \rightarrow F_0 = E/C_0$ gives the following diagram of isogenies:
\begin{center}
\begin{tikzpicture}
\draw[fill=black] (0,0) circle [radius=.1]
(2,0) circle [radius=.1]
(4,0) circle [radius=.1]
(8,0) circle [radius=.1]
(0,-1) circle [radius=.1];  

\draw[fill=black]
(5.9,0) circle [radius=.01]
(6.0,0) circle [radius=.01]
(6.1,0) circle [radius=.01];    

\draw[->](0.15,0)--(1.85,0);   
\draw[->](2.15,0)--(3.85,0);   
\draw[->](4.15,0)--(5.8,0);    
\draw[->](6.2,0)--(7.85,0);    

\draw[->](0,-0.15)--(0,-0.85); 

\node at (0,0.3) {\small$E_0$};
\node at (2,0.3) {\small$E_1$};
\node at (4,0.3) {\small$E_2$};
\node at (8,0.3) {\small$E_n$};
\node at (0,-1.3) {\small$F_0$};
\node at (-0.20,-0.5) {\footnotesize$\psi_0$};
\node at (1,0.15) {\footnotesize$\phi_0$};
\node at (3,0.15) {\footnotesize$\phi_1$};
\node at (5,0.15) {\footnotesize$\phi_2$};
\node at (7,0.15) {\footnotesize$\phi_{n-1}$};
\end{tikzpicture}
\end{center}
and for each $i = 0,\dots,n-1$ there exists a unique $\phi_{i}': F_{i} \rightarrow F_{i+1}$
with kernel $\psi_{i}(\ker(\phi_i))$ such that the following diagram commutes:
\begin{center}
\begin{tikzpicture}
\draw[fill=black]
(0,0) circle [radius=.1]
(2,0) circle [radius=.1]
(0,-1) circle [radius=.1]
(2,-1) circle [radius=.1];   

\draw[->](0.15,0)--(1.85,0);    
\draw[->](0,-0.15)--(0,-0.85);  
\draw[->](2,-0.15)--(2,-0.85);  
\draw[->](0.15,-1)--(1.85,-1);  

\node at (-0.4,0.3) {\small$C_{i}\subseteq E_{i}$};
\node at (2.5,0.3) {\small$E_{i+1}\supseteq C_{i+1}$};

\node at (0.0,-1.3) {\small$F_{i}$};
\node at (2.0,-1.3) {\small$F_{i+1}$};

\node at (1,0.2) {\footnotesize$\phi_{i}$};
\node at (-0.2,-0.5) {\footnotesize$\psi_{i}$};
\node at (1.7,-0.5) {\footnotesize$\psi_{i+1}$};
\node at (1,-0.80) {\footnotesize$\phi_{i}'$};
\end{tikzpicture}
\end{center}
The isogenies $\psi_i: E_i \rightarrow F_i$ induce orientations
$\iota_i': \OO_i' \rightarrow \End(F_i)$.
This construction motivates the following definition.
\begin{definition}
An {\it $\ell$-ladder} of length $n$ and degree $q$ is a commutative diagram
of $\ell$-isogeny chains $(E_i,\phi_i)$ and $(F_i,\phi_i')$ of length $n$
connected by $q$-isogenies $(\psi_i:E_i\rightarrow F_i)$:
\begin{center}
\begin{tikzpicture}
\draw[fill=black] (0,1) circle [radius=.1]
(2,1) circle [radius=.1]
(4,1) circle [radius=.1]
(7,1) circle [radius=.1]
(0,0) circle [radius=.1]
(2,0) circle [radius=.1]
(4,0) circle [radius=.1]
(7,0) circle [radius=.1];  

\draw[fill=black] (5.425,0) circle [radius=.01]
(5.50,0) circle [radius=.01]
(5.575,0) circle [radius=.01]
(5.425,1) circle [radius=.01]
(5.50,1) circle [radius=.01]
(5.575,1) circle [radius=.01];    
\draw[->](0.15,1)--(1.85,1);
\draw[->](2.15,1)--(3.85,1);
\draw[->](4.15,1)--(5.35,1);
\draw[->](5.65,1)--(6.85,1);
\draw[->](0.15,0)--(1.85,0);
\draw[->](2.15,0)--(3.85,0);
\draw[->](4.15,0)--(5.35,0);
\draw[->](5.65,0)--(6.85,0);
\draw[->](0,0.85)--(0,0.15);
\draw[->](2,0.85)--(2,0.15);
\draw[->](4,0.85)--(4,0.15);
\draw[->](7,0.85)--(7,0.15);
\node[anchor=south] at (0,1) {\small$E_0$};
\node[anchor=south] at (2,1) {\small$E_1$};
\node[anchor=south] at (4,1) {\small$E_2$};
\node[anchor=south] at (7,1) {\small$E_n$};
\node[anchor=north] at (0,-0.01) {\small$F_0$};
\node[anchor=north] at (2,-0.01) {\small$F_1$};
\node[anchor=north] at (4,-0.01) {\small$F_2$};
\node[anchor=north] at (7,-0.01) {\small$F_n$};
\node at (1,1.2) {\footnotesize$\phi_0$};
\node at (3,1.2) {\footnotesize$\phi_1$};
\node at (4.75,1.2) {\footnotesize$\phi_2$};
\node at (6.30,1.2) {\footnotesize$\phi_{n-1}$};
\node at (1,-0.2) {\footnotesize$\phi_0'$};
\node at (3,-0.2) {\footnotesize$\phi_1'$};
\node at (4.75,-0.2) {\footnotesize$\phi_2'$};
\node at (6.30,-0.2) {\footnotesize$\phi_{n-1}'$};

\node at (-0.2,0.5) {\footnotesize$\psi_0$};
\node at (1.8,0.5) {\footnotesize$\psi_1$};
\node at (3.8,0.5) {\footnotesize$\psi_2$};
\node at (6.8,0.5) {\footnotesize$\psi_{n}$};

\end{tikzpicture}
\end{center}
We also refer to an $\ell$-ladder of degree $q$ as a {\it $q$-isogeny} of
$\ell$-isogeny chains, which we express as $\psi: (E_i,\phi_i) \rightarrow
(F_i,\phi_i')$.

We say that an $\ell$-ladder is ascending (or descending, or horizontal)
if the $\ell$-isogeny chain $(E_i,\phi_i)$ is ascending (or descending,
or horizontal, respectively).
We say that the $\ell$-ladder is {\it level} if $\psi_0$ is a horizontal
$q$-isogeny. If the $\ell$-ladder is descending (or ascending), then we
refer to the length of the ladder as its {\it depth} (or, respectively,
as its {\it height}).
\end{definition}

\begin{lemma}\label{lemma:level}
An $\ell$-ladder $\psi: (E_i,\phi_i) \rightarrow (F_i,\phi_i')$ of oriented elliptic
curves is level if and only if $\End((E_i,\iota_i))$ is isomorphic to
$\End((F_i,\iota_i'))$ for all $0\leq i \leq n$.
In particular, if the $\ell$-ladder is level, then $(E_i,\phi_i)$ is descending
(or ascending, or horizontal) if and only if $(F_i,\phi_i')$ is descending
(or ascending, or horizontal).
\end{lemma}

\noindent{\bf Remark.}
In the sequel we will assume that $E_0$ is oriented by a maximal order $\OO_K$.
In Section~\ref{section:ActionClassGroup} we investigate using the effective
horizontal isogenies of $E_0$ to derive an effective class group action, and
introduce a modular version of this action in Section~\ref{section:Modularisogenies}.
Walking down a descending isogeny chain, each elliptic curve will be oriented
by an order of decreasing size and the final elliptic curve, which will be our
final object of study, will have an orientation by an order of large index in
$\OO_K$ with action by a large class group.

Since the supersingular $\ell$-isogeny graph is connected, every supersingular
elliptic curve admits an $\ell$-isogeny chain back to a curve oriented by any
given maximal order $\OO_K$, so such a construction exists for any supersingular
elliptic curve.

\section{Oriented curves and class group action}
\label{section:ActionClassGroup}

Let $\mathrm{SS}(p)$ denote the set of supersingular elliptic curves over
$\overline{\FF}_p$ up to isomorphism, and let $\mathrm{SS}_\OO(p)$ be the set
of $\OO$-oriented supersingular elliptic curves up to $K$-isomorphism over
$\overline{\FF}_p$, and denote the subset of primitive $\OO$-oriented
curves by $\mathrm{SS}_\OO^{pr}(p)$.

\subsection*{Class group action}

The set $\mathrm{SS}_\OO(p)$ admits a transitive group action:
\begin{center}
\begin{tikzpicture}
\node[anchor=east] at (0.80,0.66) {$\Cl(\OO) \times \mathrm{SS}_\OO(p)$};
\node[anchor=west] at (4,0.66) {$\mathrm{SS}_\OO(p)$};
\draw[->](0.90,0.66)--(3.90,0.66);
\node[anchor=east] at (0,0) {$(\left[\aa\right],E)$};
\node[anchor=west] at (4,0) {$\left[\aa\right] \cdot E=E/E[\mathfrak{a}]$};
\draw[|->](0.10,0)--(3.90,0);
\end{tikzpicture}
\end{center}
where $\mathfrak{a}$ is any representative ideal coprime to the index
$[\OO_K:\OO]$ so that the isogeny $E \rightarrow E/E[\mathfrak{a}]$
is horizontal.
When restricted to primitive $\OO$-oriented curves, we obtain the following
classical result, extending the standard result for CM elliptic curves.

\begin{theorem}
\label{theorem:oriented-action}
The class group $\Cl(\OO)$ acts faithfully and transitively on the set of
$\OO$-isomorphism classes of primitive $\OO$-oriented elliptic curves.
\end{theorem}

\noindent
In particular, for fixed primitive $\OO$-oriented $E$, we hence obtain a
bijection of sets:
\begin{center}
\begin{tikzpicture}
\node[anchor=east] at (0.25,0.66) {$\Cl(\OO)$};
\node[anchor=west] at (2,0.66) {$\mathrm{SS}_\OO^{pr}(p)$};
\draw[->](0.35,0.66)--(1.90,0.66);
\node[anchor=east] at (0,0) {$\left[\aa\right]$};
\node[anchor=west] at (2,0) {$\left[\aa\right] \cdot E$};
\draw[|->](0.10,0)--(1.90,0);
\end{tikzpicture}
\end{center}
For any ideal class $[\mathfrak{a}]$ and generating set $\{\qq_1,\dots,\qq_r\}$
of small primes, coprime to $[\OO_K:\OO]$, we can find an identity $[\mathfrak{a}]
= [\qq_1^{e_1}\cdot\ldots\cdot\qq_r^{e_r}]$, in order to compute the action via
a sequence of low-degree isogenies.

For an ordinary $\ell$-isogeny isogeny graph $\Gamma_\ell(E)$, the points
defined over $\FF_{p^n}$ are determined by the condition $\ZZ[\pi^n]
\subseteq \End(E)$.  Since the class numbers of orders $\OO$ in $K$ are
unbounded, the previous theorem implies that the oriented supersingular
graphs are infinite.  While all supersingular curves and isogenies can be
defined over $\FF_{p^2}$, we can use the inclusion of an order $\OO
\subset \End(E)$ to restrict to a finite subgraph.

\begin{corollary}
Let $(E,\iota)$ be a $K$-oriented elliptic curve. The $\ell$-isogeny graph
$\Gamma_\ell(E,\iota)$ is an infinite graph which is the union of the finite
subgraphs whose vertices are restricted to $\mathrm{SS}_{\OO}(p)$ for an
order $\OO$ in $K$.
\end{corollary}

The subrings $\OO_n = \ZZ + \ell^n\OO$ are a linearly ordered family which
serve to bound the depth of $K$-oriented curves relative to a curve at the
surface with orientation by an $\ell$-maximal order $\OO$.

\subsection*{On vortices and whirlpools}

Instead of considering the union of different isogeny graphs as in Couveignes~\cite{C2006}
and Rostovtsev-Stolbunov~\cite{RS2006}, we focus on a fixed prime $\ell$ and we think
of the other primes as acting on the $\ell$-isogeny graph. The resulting object is
the union of $\ell$-isogeny volcanoes mixing under the action of $\Cl(\OO)$.
This action stabilizes the subgraph at the surface (the craters) and preserves
descending paths.
This view is consistent with the construction of orientations by $\ell$-isogeny chains
(paths in the $\ell$-isogeny graph) anchored at the surface, with action of the class
group determined by ladders.

\begin{definition}
A {\it vortex} is defined to be an $\ell$-isogeny subgraph whose vertices
are isomorphism classes of $\OO$-oriented elliptic curves with $\ell$-maximal
endomorphism ring, equipped with the action of $\Cl(\OO)$.
A {\it whirlpool} is defined to be a complete $\ell$-isogeny graph of
$K$-oriented elliptic curves whose subgraphs of $\OO_n$-oriented classes
are acted on by $\Cl(\OO_n)$.
\end{definition}
\begin{figure}[ht!]
	\centering
	\includegraphics[height=3.5cm]{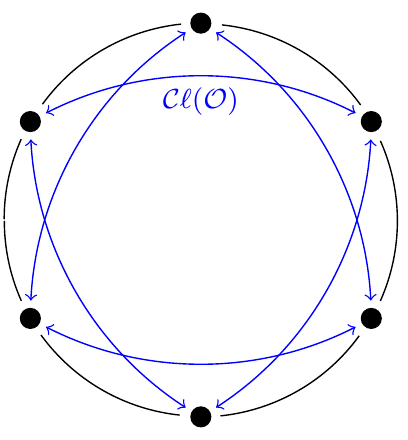}
\caption{A {\it vortex} consists of $\ell$-isogeny cycles at the surface
acted on by the class group $\Cl(\OO)$ of an $\ell$-maximal order $\OO$.}
\end{figure}

\begin{figure}[ht]
\centering
\includegraphics[width=0.44\textwidth]{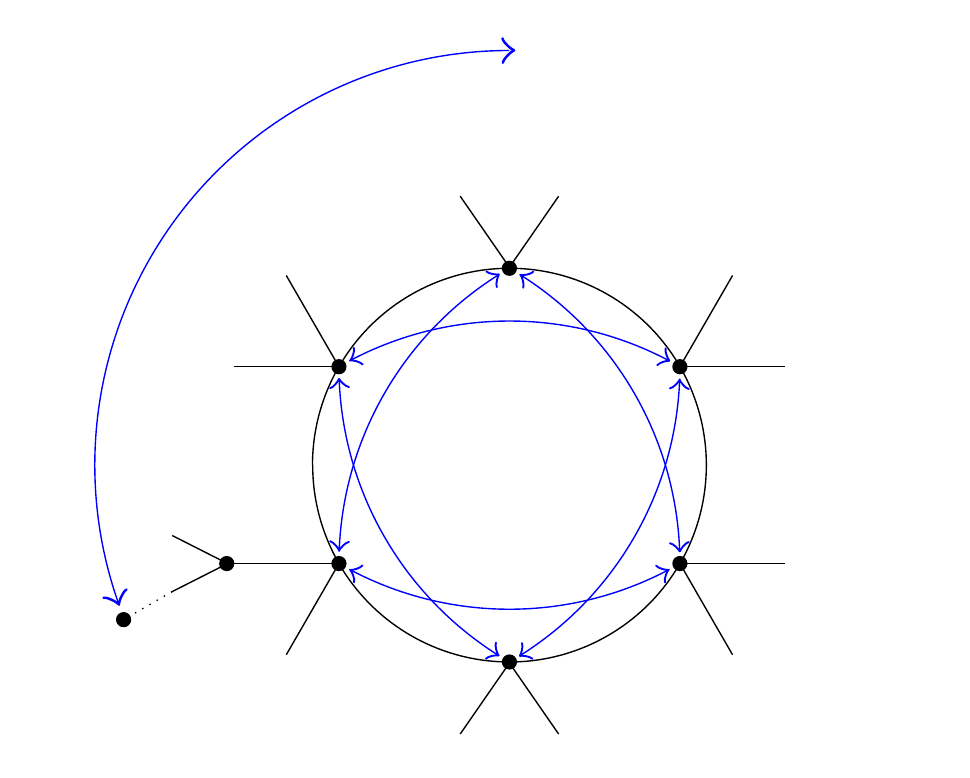}\\
\caption{A {\it whirlpool} is an $\ell$-isogeny graph equipped with compatible
actions on its subgraphs by $\Cl(\OO_n)$.  The depicted $4$-regular graph arises
from $\ell = 3$, and the cycle length is the order of a prime over $\ell$ in
the $\ell$-maximal order.}
\end{figure}

The underlying graph of a whirlpool is composed of multiple connected
components, with the class group acting transitively on components
with the same $\ell$-maximal order of its vortex.
The existence of multiple components of $\ell$-volcanoes is studied in
\cite{MSTTV2007} and \cite{FM2002}, where the set of $\ell$-volcanoes
is called an $\ell$-cordillera.
A general whirlpool can be depicted as in Figure~\ref{Figure:MultipleWhirlpool},
as an $\ell$-cordillera (black lines) acted on by the class group, as
represented by colored arrows.
\begin{figure}[H]
\centering
\includegraphics[width=0.5\textwidth]{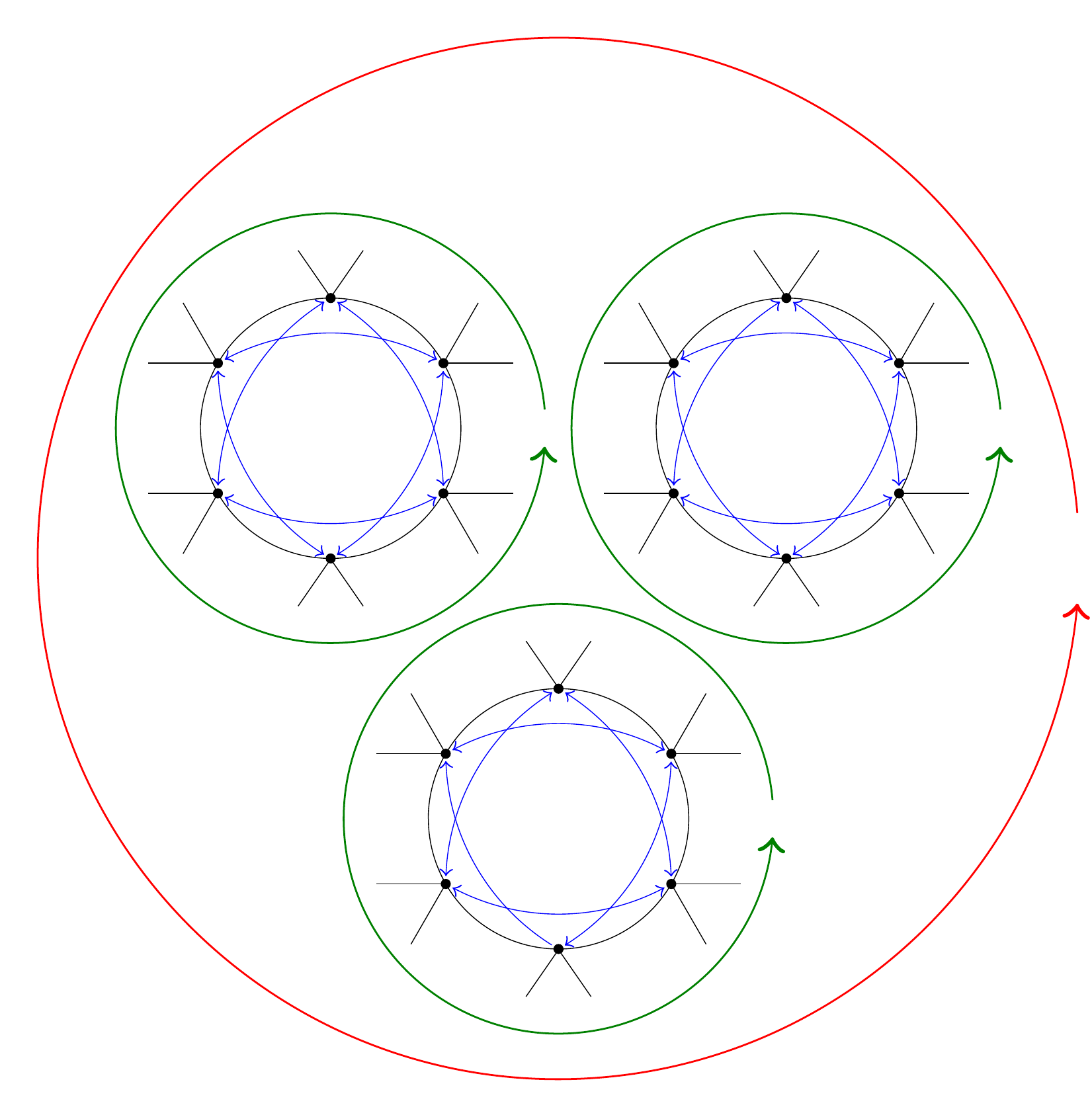}
\caption{An $\ell$-isogeny graph of a whirlpool may have multiple components.
The action depicts the subgraph acted on by a class group $\Cl(\OO)$ of order
$18$, in which $\ell = 3$ has order six, such as for discriminants $-1691$,
$-2291$, and $-2747$.}
\label{Figure:MultipleWhirlpool}
\end{figure}

\subsection*{Whirlpool examples}

We give examples of both ordinary and supersingular whirlpool structures
of $\ell$-isogeny graphs with induced class group actions.

\begin{example}
Let $E/\FF_{353}$ be a ordinary elliptic curve with $344$ rational points,
and consider the subgraph of $\Gamma_2(E)$ of curves defined over $\FF_{353}$.
The ring $\ZZ[\pi]$ generated by Frobenius $\pi$ has index~$2$ in the
maximal order $\OO_K \isom \ZZ[\sqrt{-82}]$ of class number~$4$.
The set of $j$-invariants of such curves at the surface is
$
\{160, 230, 270, 298\},
$
and the $j$-invariants of curves at depth~$1$ are
$
\{66, 182, 197, 236, 253, 264, 304, 330\}.
$

This graph, depicted in Figure~\ref{Figure:2cordillera}, consists
of two $2$-volcanoes, and hence the whirlpool consists of two components
permuted by the transitive action of $\Cl(\ZZ[\pi])$.
\begin{figure}[H] 
	\centering
	\includegraphics[height=2cm]{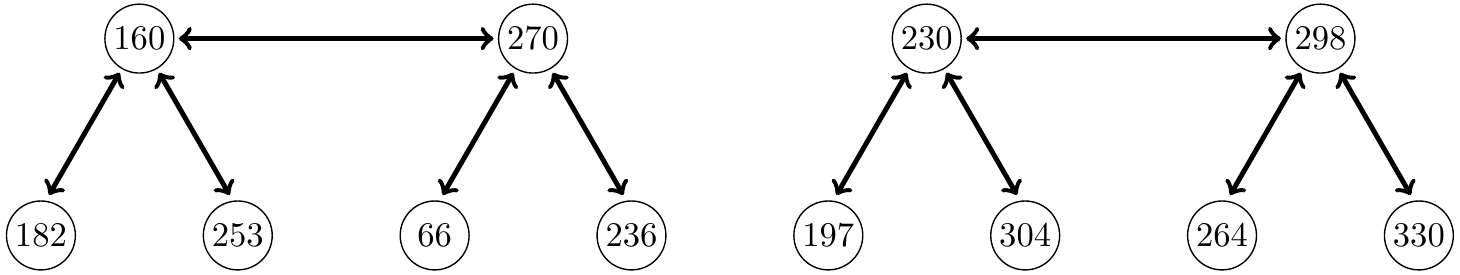}
	\caption{A $2$-cordillera.}\label{Figure:2cordillera}
	\vspace*{-0.2cm}
\end{figure}
Figure~\ref{Figure:2cordillerawhirlpool} represents the whirlpool, with blue lines
indicating the $7$-isogenies and red lines corresponding to the $13$-isogenies.

\begin{figure}[ht]
\centering
\includegraphics[width=\textwidth]{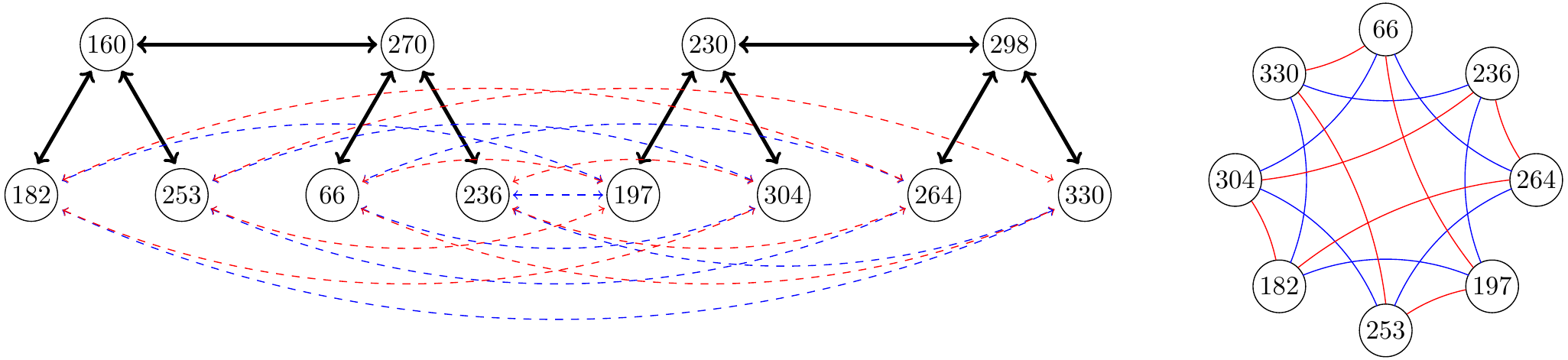}
\caption{A whirlpool with two components.}\label{Figure:2cordillerawhirlpool}
 \vspace*{-0.2cm}
\end{figure}
\end{example}

\begin{example}
Let $E_0/\FF_{71}$ be the supersingular elliptic curve with $j(E) = 0$,
oriented by the order $\OO_K = \ZZ[\omega]$, where $\omega^2 + \omega + 1 = 0$.
The unoriented $2$-isogeny graph is the finite graph:
\begin{center}
\includegraphics{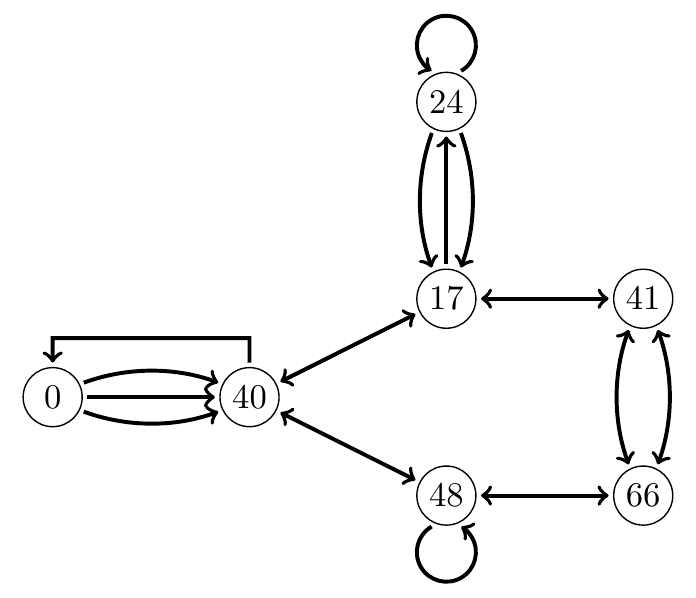}
\end{center}
The orietation by $K = \QQ[\omega]$ differentiates vertices in the descending
paths from $E_0$, determining an infinite graphy shown here to depth~$4$:
\begin{center}
\includegraphics{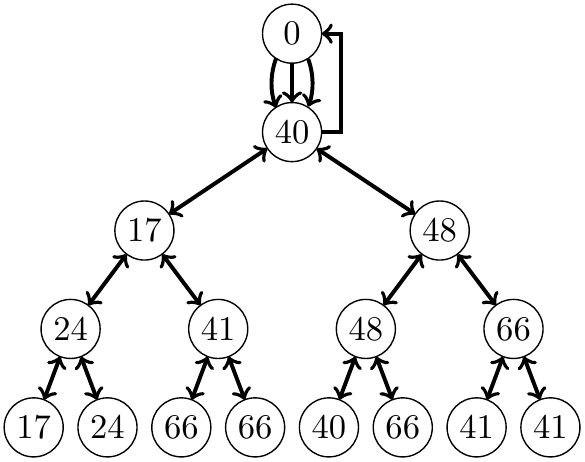}
\end{center}
Consider the descending path along vertex $j$-invariants $(0,40,17,41,66)$,
and let $\pp_7$ be a prime over the split prime $7$.  Since $\DD_K = -3$ and
$\DD_1 = \disc(\OO_1) = -12$ are of class number one, $\pp_7 \sim 1$, and the $7$-isogenous
chain is likewise of the form $(0,40,\dots)$.

At depth~$2$, the class number of $\OO_2$ of discriminant $-48$ is $2$,
and a Minkowski reduction of $\pp_7$ is an equivalent prime $\pp_3$ over $3$.
In particular, this prime is nonprincipal of order~$2$, so the image chain
extends $(0,40,48,\dots)$.

At depth~$3$, the class number of $\OO_3$ is $4$, and $\pp_7 \sim \bar{\pp}_7$
are primes of order~$2$ in the class group, hence the two $7$-isogenies are
to the same chain $(0,40,48,48,\dots)$.  Finally at depth~$4$ we
differentiate the two primes $\pp_7$ and $\bar{\pp}_7$ in $\OO_4$ each of
order~$4$. The two extensions $(0,40,48,48,66)$ and $(0,40,48,48,40)$, each
of which corresponds to one of the primes over $7$.  For a choice of prime
$\pp_7$ we have thus determined the following ladder inducing the action
of $\pp_7$ on the $\ell$-isogeny chain.

\begin{center}
\includegraphics{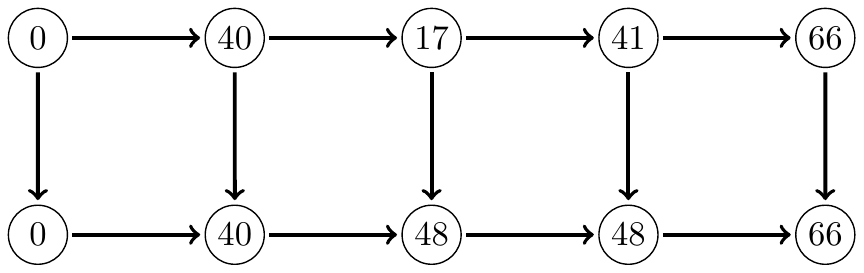}
\end{center}

\end{example}

\subsection*{The forgetful map to unoriented isogeny graphs}

In this section we address the extent of non-injectivity of the forgetful map
from oriented curves in the infinite oriented supersingular $\ell$-isogeny
graphs to the finite supersingular graph.

By Theorem~\ref{theorem:oriented-action}, we have a bijection (isomorphism
of sets with $\Cl(\OO)$-action):
$$
\Cl(\OO) \cong \SS_\OO^{pr}(\OO) \subseteq \SS_\OO(p)
$$
determined by any choice of base point.
On the other hand, for a descending chain of imaginary quadratic orders
of index $\ell$,
$$
\OO_K = \OO_0 \supset \OO_1 \supset \cdots \supset \OO_i \supset \cdots
$$
determined by a descending $\ell$-isogeny chain, the class numbers satisfy
the geometric growth $h(\OO_{i+1}) = \ell h(\OO_i)$ for all $i \ge 1$.
In particular, the inclusion $\OO_{i+1} \subset \OO_i$ determines an
inclusion $\SS_{\OO_i}(p) \subset \SS_{\OO_{i+1}}(p)
 = \SS_{\OO_i}(p) \cup \SS_{\OO_{i+1}}^{pr}(p)$.
Consequently we have an unbounded chain of sets
$$
\SS_{\OO_K}(p) \subset \SS_{\OO_1}(p) \subset \cdots \subset \SS_{\OO_i}(p) \subset \cdots
$$
equipped with forgetful maps $\SS_{\OO_i}(p) \rightarrow \SS(p)$
sending the $\OO_i$-isomorphism class $[(E,\OO_i)]$ to the 
isomorphism class $[E]$ determined by the $j$-invariant $j(E)$.

This motivates the questions of when the map $\SS_{\OO_i}(p) \rightarrow \SS(p)$
and its restriction to $\SS_{\OO_i}^{pr}(p)$ are injective, and when these
maps are surjective.
We adopt the notation $H(p)$ for the cardinality $|\SS(p)|$ of supersingular curves,
denote by $X_i$ the image of $\SS_{\OO_i}(p)$ in $\SS(p)$ and write $Y_i$ for the image of $\SS_{\OO_i}^{pr}(p)$.
Moreover we  write $\lambda_i = \log_p(|\DD_i|)$ where $\DD_i = \ell^{2i}\DD_K
= \disc(\OO_i)$.  With this notation Figure~\ref{figure:forgetful10}
and Figure~\ref{figure:forgetful12} give tables of values for $|Y_i|$, $|X_i|$,
and $\lambda_i$, for primes of $10$ and $12$ bits respectively, depicting the
boundary line for injectivity at $\lambda_i = 1$ and the critical line for surjectivity
at $\lambda_i = 2$. We conclude this section with a general proposition,
which follows from the following algebraic lemma, in order to justify the
injectivity bound.

\begin{lemma}
Let $\alpha_1$ and $\alpha_2$ be elements of a maximal quaternion order
in a quaternion algebra over $\QQ$ ramified at a prime $p$.
Set $\DD_i = \disc(\ZZ[\alpha_i])$ for $i \in \{1,2\}$,
and define $\omega$ to be the commutator
$
[\alpha_1,\alpha_2] = \alpha_1\alpha_2 - \alpha_2\alpha_1.
$
Then $\omega$ satisfies $\Tr(\omega) = 0$,
$
\Nr(\omega) = (\DD_1\DD_2-T^2)/4 \mbox{ where }
T = 2\Tr(\alpha_1\alpha_2) - \Tr(\alpha_1)\Tr(\alpha_2),
$
and $\Nr(\omega) \equiv 0 \bmod p$.
\end{lemma}

\begin{proof}
The equality $\Tr(\omega) = 0$ follows from the relation $\Tr(\alpha_1\alpha_2)
= \Tr(\alpha_2\alpha_1)$ and linearity of the reduced trace. The expression
for the reduced norm $\Nr(\omega)$ is an elementary calculation. The congruence
$\Nr(\omega) = 0 \bmod p$ holds since the unique maximal ideal $\mathfrak{P}$
over $p$ in the quaternion order is the subset of elements $\alpha$ with 
$\Nr(\alpha) \equiv 0 \bmod p$, and the quotient by $\mathfrak{P}$ is isomorphic
to the (commutative) finite field $\FF_{p^2}$.
Hence $\alpha_1\alpha_2 \equiv \alpha_2\alpha_1 \bmod \mathfrak{P}$ which 
implies $\omega \bmod \mathfrak{P} = 0$, from which $\Nr(\omega) \equiv 0 \bmod p$
holds.
\end{proof}

\begin{proposition}
Let $\OO$ be an imaginary quadratic order of discrminant $\DD$ and $p$ a prime which
is inert in $\OO$.  If $|\DD| < p$, then the map $\SS_\OO(p) \rightarrow \SS(p)$ is
injective.
\end{proposition}

\begin{proof}
If the map is not injective, there exists a supersingular elliptic curve
$E/\overline{\FF}_p$, such that $\End(E)$ admits distinct embeddings
$\iota_i :\OO = \ZZ[\alpha] \rightarrow \End(E)$, for $i \in \{1,2\}$.
Let $\alpha_i = \iota_i(\alpha)$ and set $\omega = [\alpha_1,\alpha_2]$.
By the previous lemma, we have
$$
\Nr(\omega) = \frac{\DD^2-T^2}{4} \equiv 0 \bmod p.
$$
Since $p$ is prime, and $T\equiv \DD \bmod 2$, we have either
$|\DD|-|T| \equiv 0 \bmod 2p$ or $|\DD|+|T| \equiv 0 \bmod 2p$.
Moreover, since $\End(E)$ is an order in a definite quaternion
algebra, we have $\Nr(\omega) > 0$, hence $|T| < |\DD|$.
It follows that 
$
2p \le |\DD| + |T| \le 2|\DD|,
$
and hence $p \le |\DD|$.
As a consequence, we conclude that if the map is injective, then $|\DD| < p$.
\end{proof}

\begin{figure}[ht!]
\centering
\begin{tabular}{*{6}{@{\,}c@{\,}|}}
\multicolumn{6}{c}{$p=1013$}\\
    $i$ & $h(O_i)$ & $|Y_i|$ & $|X_i|$ & $H(p)$ & $\lambda_i$\\ \hline
     1  &   1  &    1   &    1   &  85  & 0.3590 \\
     2  &   2  &    2   &    3   &  85  & 0.5593 \\
     3  &   4  &    4   &    7   &  85  & 0.7596 \\
     4  &   8  &    8   &   15   &  85  & 0.9599 \\ \hline
     5  &  16  &   16   &   29   &  85  & 1.1603 \\
     6  &  32  &   26   &   47   &  85  & 1.3606 \\
     7  &  64  &   43   &   66   &  85  & 1.5609 \\
     8  & 128  &   70   &   82   &  85  & 1.7612 \\
     9  & 256  &   79   &   85   &  85  & 1.9615 \\ \hline
    10  & 512  &   83   &   85   &  85  & 2.1618
\end{tabular}
\quad
\begin{tabular}{*{6}{@{\,}c@{\,}|}}
\multicolumn{6}{c}{$p=1019$}\\
    $i$ & $h(O_i)$ & $|Y_i|$ & $|X_i|$ & $H(p)$ & $\lambda_i$\\ \hline
     1  &   1  &    1   &    1   &  86  & 0.3587 \\
     2  &   2  &    2   &    3   &  86  & 0.5588 \\
     3  &   4  &    4   &    7   &  86  & 0.7590 \\
     4  &   8  &    8   &   15   &  86  & 0.9591 \\ \hline
     5  &  16  &   15   &   30   &  86  & 1.1593 \\
     6  &  32  &   29   &   49   &  86  & 1.3594 \\
     7  &  64  &   46   &   69   &  86  & 1.5595 \\
     8  & 128  &   64   &   81   &  86  & 1.7597 \\
     9  & 256  &   83   &   84   &  86  & 1.9598 \\ \hline
    10  & 512  &   86   &   86   &  86  & 2.1600
\end{tabular}
\caption{Sizes of images of oriented classes mapping to supersingular curves}
\label{figure:forgetful10}
\end{figure}

\begin{figure}[ht!]
\centering
\begin{tabular}{*{6}{@{\,}c@{\,}|}}
\multicolumn{6}{c}{$p=4079$}\\
    $i$ & $h(O_i)$ & $|Y_i|$ & $|X_i|$ & $H(p)$ & $\lambda_i$\\ \hline
     1  &   1  &    1   &    1   & 341  & 0.2988 \\
     2  &   2  &    2   &    3   & 341  & 0.4656 \\
     3  &   4  &    4   &    7   & 341  & 0.6323 \\
     4  &   8  &    8   &   15   & 341  & 0.7991 \\
     5  &  16  &   16   &   31   & 341  & 0.9658 \\ \hline
     6  &  32  &   31   &   62   & 341  & 1.1326 \\
     7  &  64  &   61   &  113   & 341  & 1.2993 \\
     8  & 128  &  111   &  196   & 341  & 1.4661 \\
     9  & 256  &  180   &  276   & 341  & 1.6328 \\
    10  & 512  &  258   &  326   & 341  & 1.7996 \\
    11  &1024  &  318   &  340   & 341  & 1.9663 \\ \hline
    12  &2048  &  340   &  341   & 341  & 2.1331
\end{tabular}
\quad
\begin{tabular}{*{6}{@{\,}c@{\,}|}}
\multicolumn{6}{c}{$p=4091$}\\
    $i$ & $h(O_i)$ & $|Y_i|$ & $|X_i|$ & $H(p)$ & $\lambda_i$\\ \hline
     1  &   1  &    1   &    1   & 342  & 0.2987 \\
     2  &   2  &    2   &    3   & 342  & 0.4654 \\
     3  &   4  &    4   &    7   & 342  & 0.6321 \\
     4  &   8  &    8   &   15   & 342  & 0.7988 \\
     5  &  16  &   16   &   31   & 342  & 0.9655 \\ \hline
     6  &  32  &   30   &   59   & 342  & 1.1322 \\
     7  &  64  &   59   &  110   & 342  & 1.2989 \\
     8  & 128  &  107   &  182   & 342  & 1.4656 \\
     9  & 256  &  186   &  263   & 342  & 1.6323 \\
    10  & 512  &  266   &  326   & 342  & 1.7990 \\
    11  &1024  &  314   &  341   & 342  & 1.9657 \\ \hline
    12  &2048  &  339   &  342   & 342  & 2.1323
\end{tabular}
\caption{Sizes of images of oriented classes mapping to supersingular curves}
\label{figure:forgetful12}
\end{figure}

\section{Modular isogenies}\label{section:Modularisogenies}

In this section we consider the way in which we effectively represent
and compute isogenies.  With the view to oriented isogenies, we focus
on horizontal isogenies with kernel $E[\qq]$, where $E$ is a primitive
$\OO$-oriented elliptic curve and $\qq$ a prime ideal of $\iota(\OO)$.
In what follows we suppress $\iota$ and identify $\OO$ with $\iota(\OO)$.

\subsection*{Effective endomorphism rings and isogenies}

We say a subring of $\End(E)$ is effective if we have explicit polynomial
or rational functions which represent its generators.
The subring $\ZZ$ in $\End(E)$ is thus effective.
Examples of effective imaginary quadratic subrings $\OO \subset \End(E)$, are the
subring $\OO = \ZZ[\pi]$ generated by Frobenius, for either an ordinary
elliptic curve, or a supersingular elliptic curve defined over $\FF_p$,
or an elliptic curve obtained by CM construction for an order $\OO$ of
small discriminant (in absolute value).

In the Couveignes~\cite{C2006} or the Rostovtsev-Stolbunov~\cite{RS2006}
constructions, or in the CSIDH protocol~\cite{CLMPR2018}, one works with
the ring $\OO = \ZZ[\pi]$.
The disadvantage is that for large finite fields, the class group of $\OO$
is large and the primes $\qq$ in $\OO$ have no small degree elements.
For large $p$ and small $q$, the smallest degree element of a prime $\qq$
of norm $q$ is the endomorphism $[q]$, of degree $q^2$.
The division polynomial $\psi_q(x)$, which cuts out the torsion group
$E[q]$, is of degree $(q^2-1)/2$. Consequently factoring $\psi_q(x)$ to
find the kernel polynomial (see Kohel~\cite[Chapter~2]{Kohel1996})
of degree $(q-1)/2$ for $E[\qq]$ is relatively expensive.
As a result, in the SIDH protocol~\cite{JDF2011}, the ordinary
protocol of De~Feo, Smith, and Kieffer~\cite{DFKS2018}, or the CSIDH
protocol~\cite{CLMPR2018}, the curves are chosen such that the points
of $E[\qq]$ are defined over a small degree extension $\kappa/k$,
particularly $[\kappa/k] \in \{1,2\}$, and working with rational
points in $E(\kappa)$.

In the OSIDH protocol outlined below, we propose the use of an effective
CM order $\OO_K$ of class number $1$.  In particular every prime $\qq$
of norm $q$ is generated by an endomorphism of the minimal degree $q$.
For example we may take $\OO_K$ to be the Eisenstein or Gaussian integers
of discriminant $-3$ or $-4$, generated by an automorphism.
The kernel polynomial of degree $(q-1)/2$ can be computed directly
without need for a splitting field for $E[\qq]$, and the computation
of a generator isogeny is a one-time precomputation. Using an analog
of the construction of division polynomials, the computation of the
kernel polynomial requires $O(q)$ field operations.

\subsection*{Push forward isogenies}
The extension of an isogeny (or, as we will see in the next section, of an endomorphism) of $E_0$ to an $\ell$-isogeny chain
$(E_i,\phi_i)$ reduces to the construction of a ladder.  At each step
we are given $\phi_i: E_{i} \rightarrow E_{i+1}$ and $\psi_i: E_{i}
\rightarrow F_{i}$ of coprime degrees, and need to compute
$$
\psi_{i+1}: E_{i+1} \rightarrow F_{i+1}
\mbox{ and }
\phi_{i}': F_{i} \rightarrow F_{i+1}.
$$
Rather than working with elliptic curves and isogenies, we construct
the oriented graphs directly as points on a modular curve linked by
modular correspondences defined by modular polynomials.

\subsection*{Modular curves and isogenies}

The use of modular curves for efficient computation of isogenies has
an established history (see Elkies~\cite{Elkies1998}).
For this purpose we represent isogeny chains and ladders as finite
sequences of points on the modular curve $\XX = X(1)$ preserving the
relations given by a modular equation.
\vspace{1mm}

We recall that the modular curve $X(1) \isom \PP^1$ classifies
elliptic curves up to isomorphism, and the function $j$ generates
its function field.  The family of elliptic curves
$$
E: y^2 + xy = x^3 - \frac{36}{(j - 1728)}x - \frac{1}{(j - 1728)}
$$
covers all isomorphism classes $j \ne 0, 12^3$ or $\infty$, such that
the fiber over $j_0 \in k$ is an elliptic curve of $j$-invariant $j_0$. The curves $y^2+y=x^3$ and $y^2=x^3+x$ deal with the cases $j=0$ and $j=1728$.

The modular polynomial $\Phi_m(X,Y)$ defines a correspondence in
$X(1) \times X(1)$ such that $\Phi_m(j(E),j(E')) = 0$ if and only
if there exists a cyclic $m$-isogeny $\phi$ 
from $E$ to $E'$, possibly over some extension field.
The curve in $X(1) \times X(1)$ cut out by $\Phi_m(X,Y) = 0$ is
a singular image of the modular curve $X_0(m)$ parametrizing such
pairs $(E,\phi)$. 
\vspace{1mm}

\noindent{\bf Remark.}
The modular curve $X(1)$ can be replaced by any genus $0$ modular
curve $\XX$ parametrizing elliptic curves with level structure.
Lifting the modular polynomials back to $\XX$ of higher level
(but still genus 0) has an advantage of reducing the coefficient
size of the corresponding modular polynomials $\Phi_m(X,Y)$.

In the case of CSIDH, the authors use $\XX = X_0(4)$, with a modular
function $a \in k(X_0(4))$ to parametrize the family of curves
$$
E: y^2 = x(x^2 + ax + 1),
$$
together with a cyclic subgroup $C \subset E$ of order~$4$, whose
generators are cut out by $x = 1$. The map $\XX \rightarrow X(1)$
is given by
$$
j = \frac{2^8(a^2-3)^3}{(a-2)(a+2)}\cdot
$$
The approach via modular isogenies of this section can be
adapted as well to the CSIDH protocol.

\ignore{
/*
In CSIDH the parameter A (of y^2 = (x^2 + A*x + 1)*x) determines a point on X_0(4).
*/
F<a1,a2> := FunctionField(ZZ,2);
Phi2 := ClassicalModularPolynomial(2);
Phi3 := ClassicalModularPolynomial(3);
Phi5 := ClassicalModularPolynomial(5);
// X_0(4):
j1 := 2^8*(X^2-3)^3/((X-2)^1*(X+2)^1);
j2 := 2^8*(Y^2-3)^3/((Y-2)^1*(Y+2)^1);
// X(2):
j1_prime := 2^4*(a1^2+12)^3/((a1-2)^2*(a1+2)^2);
j2_prime := 2^4*(a2^2+12)^3/((a2-2)^2*(a2+2)^2);
//
assert Phi2([j1,j1_prime]) eq 0;
assert Phi2([j2,j2_prime]) eq 0;
//
// Determine modular polynomials in X_0(4) x X_0(4):
//
Factorization(Phi2([j1,j2]));
Factorization(Phi3([j1,j2]));
Factorization(Phi5([j1,j2]));
}

\begin{definition}
A {\it modular $\ell$-isogeny chain} of length $n$ over $k$ is a finite sequence
$
(j_0,j_1,\dots,j_n)
$
in $k$ such that $\Phi_\ell(j_i,j_{i+1}) = 0$ for $0 \le i < n$.
A {\it modular $\ell$-ladder} of length $n$ and degree $q$ over $k$ is a pair
of modular $\ell$-isogeny chains
$$
(j_0,j_1,\dots,j_n) \mbox{ and } (j_0',j_1',\dots,j_n'),
$$
such that $\Phi_q(j_i,j_i') = 0$.
\end{definition}

Clearly an $\ell$-isogeny chain $(E_i,\phi_i)$ determines the modular
$\ell$-isogeny chain $(j_i = j(E_i))$, but the converse is equally true.

\begin{proposition}
If $(j_0,\dots,j_n)$ is a modular $\ell$-isogeny chain over $k$, and $E_0/k$
is an elliptic curve with $j(E_0) = j_0$, then there exists an $\ell$-isogeny
chain $(E_i,\phi_i)$ such that $j_i = j(E_i)$ for all $0 \le i \le n$.
\end{proposition}

Given any modular $\ell$-isogeny chain $(j_i)$, elliptic curve $E_0$ with
$j(E_0) = j_0$, and isogeny $\psi_0: E_0 \rightarrow F_0$, it follows that
 we can construct an $\ell$-ladder
$
\psi: (E_i,\phi_i) \rightarrow (F_i,\phi_i')
$
and hence a modular $\ell$-isogeny ladder.  In fact the $\ell$-ladder can
be efficiently constructed recursively from the modular $\ell$-isogeny
chain $(j_0,\dots,j_n)$ and $(j_0',\dots,j_n')$, by solving the system 
of equations
$$
\Phi_\ell(j'_i,Y)=\Phi_q(j_{i+1},Y)=0,
$$
for $Y = j_{i+1}'$.

\vspace{2mm}
\noindent{\bf Remark.}
The modular polynomial $\Phi_q(X,Y)$ is degree $q+1$ in $X$ and $Y$. 
The evaluation at $X = j \in \FF_{p^2}$ requires $O(q^2)$ field
multiplications. The subsequent gcd requires $O(\ell q)$ operations,
and these operations are repeated to depth $n$.

\section{OSIDH}\label{section:OSIDH}
We consider an elliptic curve $E_0/k$ ($k=\FF_{p^2}$) with an $\OO_K$-orientation
by an effective ring $\OO_K$ of class number~$1$, e.g.~$j = 0$ or $j = 12^3$ (for which $\OO_K=\ZZ[\zeta_3]$ or $\ZZ[i]$),
small prime $\ell$, and a descending $\ell$-isogeny chain from $E_0$ to $E = E_n$.
The $\OO_K$-orientation on $E_0$ and $\ell$-isogeny chain induces isomorphisms
$$
\iota_i : \ZZ + \ell^i\OO_K \rightarrow \OO_i \subset \End(E_i),
$$
and we set $\OO = \OO_n$.
By hypothesis on $E_0/k$ (the class number of $\OO_K$ is $1$), any horizontal isogeny $\psi_0: E_0 \rightarrow F_0$
is, up to isomorphism $F_0 \isom E_0$, an endomorphism.

For a small prime $q$, we push forward a $q$-endomorphism $\phi_0 \in \End(E_0)$,
to a $q$-isogeny $\psi : (E_i,\phi_i) \rightarrow (F_i,\phi_i')$.
\begin{center}
\begin{tikzpicture}

\draw[fill=black] (0,0.75*\sqt) circle [radius=0.1] 
(1.7*\sqt/2,0.75*\sqt+1.7/2) circle [radius=0.1] 
(1.7*\sqt,0.75*\sqt+1.7) circle [radius=0.1] 
(2*1.7*\sqt,0.75*\sqt+2*1.7) circle [radius=0.1]; 

\node at (0.15,0.95*\sqt) {$E_0$};
\node at (1.7*\sqt/2,0.75*\sqt+1.7/2+0.3) {$E_1$};
\node at (1.7*\sqt,0.75*\sqt+1.7+0.3) {$E_2$};
\node at (2*1.7*\sqt,0.75*\sqt+2*1.7+0.3) {$E_n$};

\draw[->](0.1*\sqt,0.75*\sqt+0.1)--(1.5*\sqt/2,0.75*\sqt+1.5/2);
\draw[->](1.9*\sqt/2,0.75*\sqt+1.9/2)--(1.6*\sqt,0.75*\sqt+1.6);	\draw[->](1.8*\sqt,0.75*\sqt+1.8)--(3*1.7*\sqt/2-0.1*\sqt,0.75*\sqt+3*1.7/2-0.1);
\draw[->](3*1.7*\sqt/2+0.1*\sqt,0.75*\sqt+3*1.7/2+0.1)--(2*1.7*\sqt-0.1*\sqt,0.75*\sqt+2*1.7-0.1);

\draw[dotted](3*1.7*\sqt/2-0.1*\sqt,0.75*\sqt+3*1.7/2-0.1)--(3*1.7*\sqt/2+0.1*\sqt,0.75*\sqt+3*1.7/2+0.1);

\node at (1.7*\sqt/4,2.0) {$\phi_{0}$};\node at (3*1.7*\sqt/4,2.8) {$\phi_{1}$};
\node at (5*1.7*\sqt/4,3.7) {$\phi_2$};\node at (7*1.7*\sqt/4,4.6) {$\phi_{n-1}$};

\centerarc[->,blue](-0.3,0.75*\sqt)(40:320:0.3);
\node at (-0.85,0.75*\sqt) {\small\textcolor{blue}{$\OO_K$}};

\draw[->] (0+0.1,0.75*\sqt-0.1*\sqt)--(0.75-0.1,0.1*\sqt);
\draw[fill=black] (0.75,0) circle [radius=0.1];        
\node at  (1.25,-0.3) {$F_0 = E_0$};
\node at (0.17,0.75*\sqt/2-0.1) {$\psi_0$};

\draw[fill=black] (0.75+1.7*\sqt/2,1.7/2) circle [radius=0.1]; 
\draw[->](0.75+0.1*\sqt,0.1)-- (0.75+1.7*\sqt/2-0.1*\sqt,1.7/2-0.1);
\node at  (0.75+1.7*\sqt/2,1.7/2-0.3) {$F_1$};
\node at (1.7*\sqt/2,0.2) {$\phi_{0}'$};
\draw[->](1.7*\sqt/2+0.1,0.75*\sqt+1.7/2-0.1*\sqt) --(0.75+1.7*\sqt/2-0.1,1.7/2+0.1*\sqt);
\node at (0.2+1.7*\sqt/2,0.75*\sqt/2+0.75) {$\psi_1$};

\draw[fill=black] (0.75+1.7*\sqt,1.7) circle [radius=0.1]; 
\node at  (0.75+1.7*\sqt,1.7-0.3) {$F_2$};
\draw[->](0.75+1.7*\sqt/2+0.1*\sqt,1.7/2+0.1)-- (0.75+1.7*\sqt-0.1*\sqt,1.7-0.1);
\node at (1.7*\sqt,1.0) {$\phi_{1}'$};
\draw[->](1.7*\sqt+0.1,0.75*\sqt+1.7-0.1*\sqt)--(0.75+1.7*\sqt-0.1,1.7+0.1*\sqt);
\node at (0.2+1.7*\sqt,0.75*\sqt/2+2*0.8) {$\psi_2$};

\draw[->](0.75+3*1.7*\sqt/2+0.1*\sqt,3*1.7/2+0.1)-- (0.75+2*1.7*\sqt-0.1*\sqt,2*1.7-0.1);
\draw[->](0.75+1.7*\sqt+0.1*\sqt,1.7+0.1)-- (0.75+3*1.7*\sqt/2-0.1*\sqt,3*1.7/2-0.1);
\draw[fill=black] (0.75+2*1.7*\sqt,2*1.7) circle [radius=0.1]; 
\node at  (0.75+2*1.7*\sqt,2*1.7-0.3) {$F_n$};
\node at (3*1.7*\sqt/2+0.1,1.9) {$\phi_{2}'$};\node at (2*1.7*\sqt+0.2,2.7) {$\phi_{n-1}'$};
\draw[->](2*1.7*\sqt+0.1,0.75*\sqt+2*1.7-0.1*\sqt)--(0.75+2*1.7*\sqt-0.1,2*1.7+0.1*\sqt);
\draw[dotted](0.75+3*1.7*\sqt/2-0.1*\sqt,3*1.7/2-0.1)--(0.75+3*1.7*\sqt/2+0.1*\sqt,3*1.7/2+0.1);
\node at (0.2+2*1.7*\sqt,0.75*\sqt/2+4*0.85-0.1) {$\psi_n$};

\end{tikzpicture}
\end{center}

By sending $\qq \subset \OO_K$ to $\psi_0 : E_0 \rightarrow F_0 = E_0/E_0[\qq] \isom E_0$,
and pushing forward to $\psi_n : E_n \rightarrow F_n$, we obtain the effective action of
$\Cl(\OO)$ on $\ell$-isogeny chains of length $n$ from $E_0$. In other words, the action of an ideal $\qq$ becomes non trivial while pushing it down along a descending isogeny chain due to the fact that $\qq\cap\OO_i$ becomes ``less and less principal''.

In order to have the action of $\Cl(\OO)$ cover a large portion of the supersingular
elliptic curves, we require $\ell^n \sim p$, i.e., $n \sim \log_\ell(p)$.
%
\vspace{2mm}

\noindent{\bf Recall.}
The previous estimates are based on two very important results. Observe that the number of oriented elliptic curves that we can reach after $n$ steps equals the class number $h(\OO_n)$ of $\OO_n=\ZZ+\ell^n\OO_K$. It is well-known \cite[\S 7.D]{Cox97} that:
\begin{equation}\label{equation:classnumberformula}
h(\ZZ + m\OO_K)=\frac{h(\OO_K)m}{\left[\OO_K^\times:\OO^\times\right]}\prod_{p\mid m}\left(1-\left(\frac{\Delta_K}{p}\right)\frac{1}{p}\right)
\end{equation}
where \cite[VI.3]{Cohn1980}
\[
\OO_K^\times=\begin{cases}
\{\pm 1\}~&\text{if }\Delta_K<-4\\
\{\pm1,\pm i \}&\text{if }\Delta_K=-4\\
\{\pm1,\pm\zeta_3,\pm\zeta_3^2 \}&\text{if }\Delta_K=-3
\end{cases}
~~\Rightarrow~~
\left[\OO_K^\times:\OO^\times\right]=\begin{cases}
1~&\text{if }\Delta_K<-4\\
2&\text{if }\Delta_K=-4\\
3&\text{if }\Delta_K=-3
\end{cases}
\]
On the other hand, we know that the number of supersingular elliptic curves over $\FF_{p^2}$ is given by the following formula \cite[V.4]{Silv1986}:
\[
\#\text{SS}(p)=\left[\frac{p}{12}\right]+\begin{cases}
0~~&\text{if }p\equiv1\bmod 12\\
1&\text{if }p\equiv5,7\bmod 12\\
2&\text{if }p\equiv11\bmod 12
\end{cases}
\]
Therefore, in our case
\[
h(\ell^n\OO_K)=\frac{1\cdot \ell^n}{2\text{ or }3}\left(1-\left(\frac{\Delta_K}{\ell}\right)\frac{1}{\ell}\right)=\left[\frac{p}{12}\right]+\epsilon~\Longrightarrow~p\sim\ell^n
\]

To realise the class group action, it suffices to replace the above $\ell$-ladder
with its modular $\ell$-ladder.
\vspace*{-0.5cm}
\begin{center}
\begin{tikzpicture}
\draw[fill=black] (0,0.75*\sqt) circle [radius=0.1]; 
\draw[fill=black] (1.7*\sqt/2,0.75*\sqt+1.7/2) circle [radius=0.1]; 
\draw[fill=black] (1.7*\sqt,0.75*\sqt+1.7) circle [radius=0.1]; 
\draw[fill=black] (2*1.7*\sqt,0.75*\sqt+2*1.7) circle [radius=0.1]; 

\node at (0.15,0.95*\sqt) {$j_0$};\node at (1.7*\sqt/2,0.75*\sqt+1.7/2+0.3) {$j_1$};
\node at (1.7*\sqt-0.2,0.75*\sqt+1.7+0.45) {$j_2$};
\node at (2*1.7*\sqt,0.75*\sqt+2*1.7+0.3) {$j_n$};
\draw[->](0.1*\sqt,0.75*\sqt+0.1)--(1.5*\sqt/2,0.75*\sqt+1.5/2);
\draw[->](1.9*\sqt/2,0.75*\sqt+1.9/2)--(1.6*\sqt,0.75*\sqt+1.6);
\draw[->](1.8*\sqt,0.75*\sqt+1.8)--(3*1.7*\sqt/2-0.1*\sqt,0.75*\sqt+3*1.7/2-0.1);
\draw[->](3*1.7*\sqt/2+0.1*\sqt,0.75*\sqt+3*1.7/2+0.1)--(2*1.7*\sqt-0.1*\sqt,0.75*\sqt+2*1.7-0.1);

\draw[dotted](3*1.7*\sqt/2-0.1*\sqt,0.75*\sqt+3*1.7/2-0.1)--(3*1.7*\sqt/2+0.1*\sqt,0.75*\sqt+3*1.7/2+0.1);

\node at (1.7*\sqt/4,1.9) {\small$\ell$};\node at (3*1.7*\sqt/4,2.75) {\small$\ell$};
\node at (5*1.7*\sqt/4,3.6) {\small$\ell$};\node at (7*1.7*\sqt/4,4.45) {\small$\ell$};

\centerarc[->,blue](-0.3,0.75*\sqt)(40:320:0.3);
\node at (-0.85,0.75*\sqt) {\small\textcolor{blue}{$\OO_K$}};

\draw[double equal sign distance] (0+0.1,0.75*\sqt-0.1*\sqt)--(0.75-0.1,0.1*\sqt);
\draw[fill=black] (0.75,0) circle [radius=0.1];        
\node at  (0.75,-0.35) {$j'_0$};
\node at (0.17,0.75*\sqt/2-0.1) {\small $q$};

\draw[fill=black] (0.75+1.7*\sqt/2,1.7/2) circle [radius=0.1]; 
\draw[->](0.75+0.1*\sqt,0.1)-- (0.75+1.7*\sqt/2-0.1*\sqt,1.7/2-0.1);
\node at  (0.75+1.7*\sqt/2,1.7/2-0.35) {$j'_1$};
\node at (1.7*\sqt/2,0.25) {\small$\ell$};
\draw[->](1.7*\sqt/2+0.1,0.75*\sqt+1.7/2-0.1*\sqt) --(0.75+1.7*\sqt/2-0.1,1.7/2+0.1*\sqt);
\node at (0.2+1.7*\sqt/2,0.75*\sqt/2+0.75) {\small $q$};

\draw[fill=black] (0.75+1.7*\sqt,1.7) circle [radius=0.1]; 
\draw[->](0.75+1.7*\sqt/2+0.1*\sqt,1.7/2+0.1)-- (0.75+1.7*\sqt-0.1*\sqt,1.7-0.1);
\node at (1.7*\sqt,1.1) {\small$\ell$};
\draw[->](1.7*\sqt+0.1,0.75*\sqt+1.7-0.1*\sqt)--(0.75+1.7*\sqt-0.1,1.7+0.1*\sqt);
\node at (0.2+1.7*\sqt,0.75*\sqt/2+2*0.8) {\small $q$};

\draw[->](0.75+3*1.7*\sqt/2+0.1*\sqt,3*1.7/2+0.1)-- (0.75+2*1.7*\sqt-0.1*\sqt,2*1.7-0.1);
\draw[->](0.75+1.7*\sqt+0.1*\sqt,1.7+0.1)-- (0.75+3*1.7*\sqt/2-0.1*\sqt,3*1.7/2-0.1);
\draw[fill=black] (0.75+2*1.7*\sqt,2*1.7) circle [radius=0.1]; 
\node at  (0.75+2*1.7*\sqt,2*1.7-0.35) {$j'_n$};
\node at (3*1.7*\sqt/2,1.95) {\small$\ell$};\node at (2*1.7*\sqt,2.8) {\small$\ell$};
\draw[->](2*1.7*\sqt+0.1,0.75*\sqt+2*1.7-0.1*\sqt)--(0.75+2*1.7*\sqt-0.1,2*1.7+0.1*\sqt);
\draw[dotted](0.75+3*1.7*\sqt/2-0.1*\sqt,3*1.7/2-0.1)--(0.75+3*1.7*\sqt/2+0.1*\sqt,3*1.7/2+0.1);
\node at (0.2+2*1.7*\sqt,0.75*\sqt/2+4*0.85-0.1) {\small $q$};

\draw[blue,rounded corners,rotate around={-60:(1.7*\sqt+0.75/2,0.75*\sqt/2+1.7)}] (1.7*\sqt+0.75/2-1,0.75*\sqt/2+1.7+0.3) rectangle (1.7*\sqt+0.75/2+1,0.75*\sqt/2+1.7-0.3) {};

\node at (6,0.5) {
$
\begin{cases}
\Phi_\ell(j_1,j_2)=0\\
\Phi_\ell(j'_1,Y)=0\\
\Phi_q(j_2,Y)=0
\end{cases}
$};

\draw[blue](1.7*\sqt+0.75/2+0.5,0.75*\sqt/2+1.7-0.5*\sqt)--(1.7*\sqt+0.75/2+0.5,0.75*\sqt/2+1.7-0.5*\sqt-0.81);
\draw[blue,->](4,0.5)--(4.5,0.5);
\draw[blue] (4,0.5) arc (-90:-180:0.180513627);
\end{tikzpicture}
\end{center}
At the first index for which $j_i' = j(E_i/E_i[\qq_i])$ is different
from $j_i'' = j(E_i/E_i[\bar{\qq}_i])$, that is, $[\qq_i] \ne [\bar{\qq}_i]$
in $\Cl(\OO_i)$, we can solve iteratively for $j_{i+1}'$ from $j_{i}'$
and $j_{i+1}$ using the equations:
\[
\Phi_\ell(j'_{i},Y)=\Phi_q(j_{i+1},Y)=0.
\]
The action of primes $\qq$ through $\Cl(\OO)$ can be precomputed by its
action on these initial segments which permits us to separate the action
of $\qq$ and $\bar{\qq}$, hence assures a unique solution to the above
system.


\begin{center}
\begin{tikzpicture}
\draw[fill=black] (0,0) circle [radius=.1]
(2,0) circle [radius=.1]
(4,0) circle [radius=.1]
(0,-1) circle [radius=.1]
(2,-1) circle [radius=.1]
(4,-1) circle [radius=.1];
\draw[->] (0.2,0.08) -- (1.8,0.08);
\draw[<-] (0.2,-0.08) -- (1.8,-0.08);
\draw[->] (2.2,0) -- (3.8,0);
\draw[->] (0,-0.2) -- (0,-0.8);
\draw[->] (2,-0.2) -- (2,-0.8);
\draw[->] (4,-0.2) -- (4,-0.8);
\draw[->] (0,-1.2) -- (0,-1.8);
\draw[->] (2,-1.2) -- (2,-1.8);
\draw[->] (4,-1.2) -- (4,-1.8);
\draw[<-] (0.2,-1) -- (1.8,-1);
\draw[->] (2.2,-1) -- (3.8,-1);
\node[anchor=east] at (0,0) {$E_0$};
\node[anchor=west] at (4,0) {$E_0$};
\node[anchor=north west] at (2,0) {$E_0$};
\node[anchor=east] at (0,-1) {$E_1''$};
\node[anchor=west] at (4,-1) {$E_1'$};
\node[anchor=north west] at (2,-1) {$E_1$};
\begin{scope}[yshift=0.1cm]
\centerarc[<-](2,-4)(66:114:4.442135954999579);
\end{scope}
\node at (3,0.15) {$\qq$};
\node at (1,0.2) {$\qq$};
\node at (1,-0.3) {$\bar{\qq}$};
\node at (2,0.85) {$\qq^2$};
\end{tikzpicture}
\end{center}
Thus, $E_i'\ne E_i''$ if and only if $\qq^2\cap \OO_i$ is not principal and the probability that a random ideal in $\OO_i$ is principal is $1/h(\OO_i)$. In fact, we can do better; we write $\OO_K=\ZZ[\omega]$ and we observe that if $\qq^2$ was principal, then
\[
q^2=\mathrm{N}(\qq^2)=\mathrm{N}(a+b\ell^i\omega)
\]
since it would be generated by an element of $\OO_i=\ZZ+\ell^i\OO_K$. Now
\[
\mathrm{N}(a+b\ell^i)=a^2\pm abt\ell^i+b^2s\ell^{2i}~~~~~\text{where}~~~~~\omega^2+t\omega+s=0
\]
Thus, as soon as $\ell^{2i}>q^2$ we are guaranteed that $\qq^2$ is not principal.

\subsection{A first naive protocol}\label{subsection:FirstNaiveProtocol}

We now present the OSIDH cryptographic protocol based on this construction.
We first describe a simplified version as intermediate step. The reason for
doing that is twofold. On one hand it permits us to observe how the notions
introduced so far lead to a cryptographic protocol, and on the other hand
it highlights the critical security considerations and identifies the
computationally hard problems on which the security is based.

As described at the beginning of the section, we fix a maximal order
$\OO_K$ in a quadratic imaginary field $K$ of small discriminant $\Delta_K$
and a large prime $p$ such that $\left(\frac{\Delta_K}{p}\right)\ne 1$.
Further, the two parties agree on an elliptic curve $E_0$ with effective
maximal order $\OO_K$ embedded in the endomorphism ring and a descending $\ell$-isogeny chain:
\[
E_0\longrightarrow E_1 \longrightarrow E_2\longrightarrow\cdots\longrightarrow E_n.
\]
Each constructs a power smooth horizontal endomorphism $\psi$ of $E_0$ as the product
of generators of small principal ideals in $\OO_K$.  A power smooth isogeny,
for which the prime divisors and exponents of its degree are bounded, ensures
that $\psi$ can be efficiently extended to a ladder.

\noindent{\bf Remark.} In practice, we will fix $\OO_K$ to be either the Eisenstein
integers $\ZZ[\zeta_3]$ or the Gaussian integers $\ZZ[\zeta_4] (= \ZZ[i])$. Since
the ladder is descending, we have that $\End((E_i,\iota_i)) \isom \ZZ+\ell^i\OO_K$
for all $i = 0,\dots,n$.

Alice privately chooses a horizontal power smooth endomorphism
$\psi_A = \psi_0: E_0 \to F_0=E_0$, and pushes it forward to an
$\ell$-ladder of length $n$:
\begin{center}
	\begin{tikzpicture}
	\draw[fill=black] (0,1) circle [radius=.1]
	(2,1) circle [radius=.1]
	(4,1) circle [radius=.1]
	(7,1) circle [radius=.1]
	(0,0) circle [radius=.1]
	(2,0) circle [radius=.1]
	(4,0) circle [radius=.1]
	(7,0) circle [radius=.1];  

	\draw[fill=black] (5.425,0) circle [radius=.01]
	(5.50,0) circle [radius=.01]
	(5.575,0) circle [radius=.01]
	(5.425,1) circle [radius=.01]
	(5.50,1) circle [radius=.01]
	(5.575,1) circle [radius=.01];    
	\draw[->](0.15,1)--(1.85,1);
	\draw[->](2.15,1)--(3.85,1);
	\draw[->](4.15,1)--(5.35,1);
	\draw[->](5.65,1)--(6.85,1);
	\draw[->](0.15,0)--(1.85,0);
	\draw[->](2.15,0)--(3.85,0);
	\draw[->](4.15,0)--(5.35,0);
	\draw[->](5.65,0)--(6.85,0);
	\draw[->](0,0.85)--(0,0.15);
	\draw[->](2,0.85)--(2,0.15);
	\draw[->](4,0.85)--(4,0.15);
	\draw[->](7,0.85)--(7,0.15);
	\node[anchor=south] at (0,1) {\small$E_0$};
	\node[anchor=south] at (2,1) {\small$E_1$};
	\node[anchor=south] at (4,1) {\small$E_2$};
	\node[anchor=south] at (7,1) {\small$E_n$};
	\node[anchor=north] at (0,-0.01) {\small$F_0$};
	\node[anchor=north] at (2,-0.01) {\small$F_1$};
	\node[anchor=north] at (4,-0.01) {\small$F_2$};
	\node[anchor=north] at (7,-0.01) {\small$F_n$};
	\node at (1,1.2) {$\phi_0$};
	\node at (3,1.2) {$\phi_1$};
	\node at (4.75,1.2) {$\phi_2$};
	\node at (6.30,1.2) {$\phi_{n-1}$};
	\node at (1,-0.2) {$\phi_0'$};
	\node at (3,-0.2) {$\phi_1'$};
	\node at (4.75,-0.2) {$\phi_2'$};
	\node at (6.30,-0.2) {$\phi_{n-1}'$};

	\node at (-0.3,0.5) {\footnotesize$\psi_A$};

	\end{tikzpicture}
\end{center}
By Lemma~\ref{lemma:level}, this $\ell$-ladder is level, hence $\End((E_i,\iota_i))
= \End((F_i,\iota_i'))$.

The $\ell$-isogeny chain $(F_i)$ is sent to Bob, who chooses a horizontal smooth
endomorphism $\psi_B$,
and sends the resulting $\ell$-isogeny chain $(G_i)$ to Alice.  Each applies (and, eventually, push forward) the
private endomorphism to obtain $(H_i) = \psi_B \cdot (F_i) = \psi_A \cdot (G_i)$, and
$H = H_n$ is the shared secret.

In the following picture the blue arrows correspond to the orientation chosen throughout by
Alice while the red ones represent the choice made by Bob.
\begin{center}
\begin{tikzpicture}
\draw[fill=black] (0,0) circle [radius=.1]
(0,2) circle [radius=.1]
(1,1) circle [radius=.1]
(1,3) circle [radius=.1];

\draw[blue,->] (0,0.2) -- (0,1.8);
\draw[blue] (1,1.2) -- (1,1.9);
\draw[blue,->] (1,2.1) -- (1,2.8);
\draw[red,->] (0.141421,0.141421) -- (1-0.141421,1-0.141421);
\draw[red,->] (0.141421,2+0.141421) -- (1-0.141421,3-0.141421);
\draw[->] (0.2,0) -- (2.8,0);
\draw[blue,->] (0.2,2) -- (2.8,2);
\draw[->] (1.2,3) -- (3.8,3);
\draw[red] (1.2,1) -- (2.9,1);
\draw[red,->] (3.1,1) -- (3.8,1);

\node[anchor=north] at (0,-0.1) {$E_0$};
\node[anchor=east] at (0,1.7) {$F_0$};
\node[anchor=north] at (1.25,0.95) {$G_0$};
\node[anchor=south] at (1,3) {$H_0$};

\begin{scope}[xshift=3cm]  
\draw[fill=black] (0,0) circle [radius=.1]
(0,2) circle [radius=.1]
(1,1) circle [radius=.1]
(1,3) circle [radius=.1];

\draw[blue,->] (0,0.2) -- (0,1.8);
\draw[blue] (1,1.2) -- (1,1.9);
\draw[blue,->] (1,2.1) -- (1,2.8);
\draw[red,->] (0.141421,0.141421) -- (1-0.141421,1-0.141421);
\draw[red,->] (0.141421,2+0.141421) -- (1-0.141421,3-0.141421);
\draw[->] (0.2,0) -- (2.8,0);
\draw[blue,->] (0.2,2) -- (2.8,2);
\draw[->] (1.2,3) -- (3.8,3);
\draw[red] (1.2,1) -- (2.9,1);
\draw[red,->] (3.1,1) -- (3.8,1);

\node[anchor=north] at (0,-0.1) {$E_1$};
\node[anchor=east] at (0,1.7) {$F_1$};
\node[anchor=north] at (1.25,0.95) {$G_1$};
\node[anchor=south] at (1,3) {$H_1$};

\end{scope}

\begin{scope}[xshift=6cm]  
\draw[fill=black] (0,0) circle [radius=.1]
(0,2) circle [radius=.1]
(1,1) circle [radius=.1]
(1,3) circle [radius=.1];

\draw[blue,->] (0,0.2) -- (0,1.8);
\draw[blue] (1,1.2) -- (1,1.9);
\draw[blue,->] (1,2.1) -- (1,2.8);
\draw[red,->] (0.141421,0.141421) -- (1-0.141421,1-0.141421);
\draw[red,->] (0.141421,2+0.141421) -- (1-0.141421,3-0.141421);
\draw[->] (0.2,0) -- (1.8,0);
\draw[blue,->] (0.2,2) -- (1.8,2);
\draw[->] (1.2,3) -- (2.8,3);
\draw[red,->] (1.2,1) -- (2.8,1);

\node[anchor=north] at (0,-0.1) {$E_2$};
\node[anchor=east] at (0,1.7) {$F_2$};
\node[anchor=north] at (1.25,0.95) {$G_2$};
\node[anchor=south] at (1,3) {$H_2$};
\end{scope}

\begin{scope}[xshift=8cm]  
\draw[fill=black] (0,0) circle [radius=.01]
(-0.1,0) circle [radius=.01]
(0.1,0) circle [radius=.01]
(1,3) circle [radius=.01]
(0.9,3) circle [radius=.01]
(1.1,3) circle [radius=.01];

\draw[red,fill=red] (1,1) circle [radius=.01]
(0.9,1) circle [radius=.01]
(1.1,1) circle [radius=.01];

\draw[blue,fill=blue] (0,2) circle [radius=.01]
(-0.1,2) circle [radius=.01]
(0.1,2) circle [radius=.01];

\draw[->] (0.2,0) -- (1.8,0);
\draw[blue,->] (0.2,2) -- (1.8,2);
\draw[->] (1.2,3) -- (2.8,3);
\draw[red] (1.2,1) -- (1.9,1);
\draw[red,->] (2.1,1) -- (2.8,1);
\end{scope}

\begin{scope}[xshift=10cm]  
\draw[fill=black] (0,0) circle [radius=.1]
(0,2) circle [radius=.1]
(1,1) circle [radius=.1]
(1,3) circle [radius=.1];

\draw[blue,->] (0,0.2) -- (0,1.8);
\draw[blue,->] (1,1.2) -- (1,2.8);
\draw[red,->] (0.141421,0.141421) -- (1-0.141421,1-0.141421);
\draw[red,->] (0.141421,2+0.141421) -- (1-0.141421,3-0.141421);

\node[anchor=north] at (0,-0.1) {$E_n$};
\node[anchor=east] at (0,1.7) {$F_n$};
\node[anchor=north] at (1.25,0.95) {$G_n$};
\node[anchor=south] at (1,3) {$H_n$};
\end{scope}

\end{tikzpicture}
\end{center}

\begin{center}
\begin{tabular}{L{3cm} C{4cm} C{4.1cm}}
\hline
\multicolumn{3}{l}{{\bf PUBLIC DATA:} A descending $\ell$-isogeny chain $E_0\to E_1\to\cdots\to E_n$}\\
\hline
&{\bf ALICE}&{\bf BOB}\\
Choose a smooth endomorphism of $E_0$ in $\OO_K$&
\begin{tikzpicture}[scale=0.5,baseline=0]
\draw[fill=black] (0,0.75*\sqt) circle [radius=0.1]; 
\draw[fill=black] (0.75,0) circle [radius=0.1];        
\draw[double equal sign distance] (0+0.1,0.75*\sqt-0.1*\sqt)--(0.75-0.1,0.1*\sqt);
\centerarc[->,blue](-0.3,0.75*\sqt)(40:320:0.3);
\node at (0.22,0.94*\sqt) {\footnotesize$E_0$}; \node at  (0.77,-0.39) {\footnotesize$F_0$};
\end{tikzpicture}&
\begin{tikzpicture}[scale=0.5,baseline=0]
\draw[fill=black] (0,0.75*\sqt) circle [radius=0.1]; 
\draw[fill=black] (0.75,0) circle [radius=0.1];        
\draw[double equal sign distance] (0+0.1,0.75*\sqt-0.1*\sqt)--(0.75-0.1,0.1*\sqt);
\centerarc[->,red](-0.3,0.75*\sqt)(40:320:0.3);
\node at (0.22,0.94*\sqt) {\footnotesize$E_0$}; \node at  (0.77,-0.39) {\footnotesize$G_0$};
\end{tikzpicture}\\
Push it forward to depth $n$&$\displaystyle \underbrace{F_0\to F_1\to\cdots\to F_n}_{\psi_A\tikzmark{a}}$&$\displaystyle \underbrace{G_0\to G_1\to\cdots\to G_n}_{\tikzmark{b}\psi_B}$\\
Exchange data&&\\
&$(G_i)$\tikzmark{c}&\tikzmark{d}$(F_i)$
\begin{tikzpicture}
[
remember picture,
overlay,
-latex,
color=green,
yshift=1ex,
shorten >=1pt,
shorten <=1pt,
]
\draw ({pic cs:a}) -- ({pic cs:d});
\draw ({pic cs:b}) -- ({pic cs:c});
\end{tikzpicture}
\\
Compute shared secret&Compute $\psi_A\cdot(G_i)$&Compute $\psi_B\cdot(F_i)$\\
\hline
\multicolumn{3}{l}{In the end, Alice and Bob share a new chain $E_0\to H_1\to\cdots\to H_n$}\\
\hline
\end{tabular}
\end{center}
This naive protocol reveals too much information and is susceptible to attack
by computing the endomorphism rings of the end curves $\End(E_n)$, $\End(F_n)$,
and $\End(G_n)$.
In general, the problem of computing an isogeny between two supersingular elliptic
curves $E$ and $F$ knowing $\End(E)$ is broadly equivalent to the task of computing
$\End(F)$ \cite{GV2018,EHLMP2018}.  Kohel's algorithm~\cite{Kohel1996}, and the refinement
of Galbraith~\cite{G1999}, compute several paths in the isogeny graph to find
isogenies $F \to F$. Thus, as noted in \cite{GV2018}, computing $\End(F)$ can be
reduced to finding an endomorphism $\phi:F\to F$ that is not in $\ZZ[\pi]$.

\noindent{\bf Remark.}
Observe that in SIDH and CSIDH the endomorphism ring of the starting elliptic curve
is known since the shared initial curve is chosen to have special form.
In OSIDH the situation changes: we need to find an isogeny starting from $E_n$,
and not the curve $E_0$ for which we have an explicit description of the
endomorphism ring.
However, knowing $\End(E_0)$, we can deduce at each step
\[
\ZZ+\ell\End(E_i) \isom \ZZ+\phi_i\End(E_i)\hat{\phi}_i \subset\End(E_{i+1})
\]
and thus we obtain the inclusion $\ZZ+\ell^n\End(E_0)\hookrightarrow\End(E_{n})$.

Notice that, in general, knowing the existence of a copy of an imaginary quadratic
order inside the maximal order of a quaternion algebra does not guarantee the
knowledge of the embedding as there might be many \cite[II.5]{Eic1973}.
In this case, from the knowledge of a subring $\ZZ+\ell\End(E_i)$ of finite
index $\ell^3$ we can reconstruct $\End(E_{i+1})$ step-by-step from the $\ell$-isogeny
chain $E_0\to E_1\to\ldots\to E_n$, and hence compute $\End(E_n)$.

In the naive protocol we also share the full isogeny chain $(F_i)$ (or their
$j$-invariant sequence), which allows an adversary to deduce the oriented
endomorphism ring
\[
\ZZ+\ell^n\OO_K\hookrightarrow\End(F_n)
\]
of the terminal elliptic curve $F = F_n$.  This gives enough information to
deduce $\Hom(E,F)$ and construct a representative smooth ideal in $\Cl(O)$
sending $E$ to $F$.


We observe that there is another approach to this problem which uses only properties
of the ideal class group. Suppose we have a $K$-descending $\ell$-isogeny chain
$
E_0\longrightarrow E_1\longrightarrow\ldots\longrightarrow E_n
$
with
\[
\End(E_0)\supsetneq\OO_K=\OO_0\supset\OO_1\supset\ldots\supset\OO_n\simeq\ZZ+\ell^n\OO_K
\]
This induces a sequence at the level of class groups
 \begin{center}
	\begin{tikzpicture}[>=stealth,->,shorten >=2pt,looseness=.5,auto]
	\matrix (M)[matrix of math nodes,row sep={6mm,between origins},
	column sep={6mm,between origins}]{
		\Cl(\OO_n)&[15mm]\cdots&[15mm]\Cl(\OO_i)&[15mm]\cdots&[15mm]\Cl(\OO_K)\\
		\rotatebox{90}{$\simeq$}&&\rotatebox{90}{$\simeq$}&&\rotatebox{90}{$\simeq$}\\
		\frac{\left(\OO_K/\ell^n\OO_K\right)^\times}{\overline{\OO}_K^\times\left(\ZZ/\ell^n\ZZ\right)^\times}&\cdots&\frac{\left(\OO_K/\ell^i\OO_K\right)^\times}{\overline{\OO}_K^\times\left(\ZZ/\ell^i\ZZ\right)^\times}&\cdots&\{1\}\\
	};
	\draw[->](M-1-1)--(M-1-2);
	\draw[->](M-1-2)--(M-1-3);
	\draw[->](M-1-3)--(M-1-4);
	\draw[->](M-1-4)--(M-1-5);
	\draw[->](M-3-1)--(M-3-2);
	\draw[->](M-3-2)--(M-3-3);
	\draw[->](M-3-3)--(M-3-4);
	\draw[->](M-3-4)--(M-3-5);
	\end{tikzpicture}
\end{center}
In particular, there exists a surjection
\[
\Cl(\OO_{i+1})\simeq\frac{\left(\OO_K/\ell^{i+1}\OO_K\right)^\times}{\overline{\OO}_K^\times\left(\ZZ/\ell^{i+1}\ZZ\right)^\times}\dhxrightarrow{~~~~~~~}{}\frac{\left(\OO_K/\ell^i\OO_K\right)^\times}{\overline{\OO}_K^\times\left(\ZZ/\ell^i\ZZ\right)^\times}\simeq\Cl(\OO_i)
\]
whose kernel is easily described. First, the map $\psi:\Cl(\OO_1)\to\Cl(\OO_K)$ has kernel
\begin{equation*}
\setlength{\arraycolsep}{0.5pt}
\renewcommand{\arraystretch}{1.2}
\left\{\begin{array}{l @{\quad} l l}
\FF_{\ell^2}^\times/\FF_{\ell}^\times&\text{of order }\ell+1~~~~~&\text{if }\ell\text{ is inert}\\
\left(\FF_{\ell}^\times\times\FF_{\ell}^\times\right)/\FF_{\ell}^\times&\text{of order }\ell-1&\text{if }\ell\text{ splits}\\
\left(\FF_{\ell}\left[\xi\right]\right)^\times/\FF_{\ell}^\times&\text{of order }\ell&\text{if }\ell\text{ is ramified}\\
\end{array}\right.
\end{equation*}
where $\xi^2=0$ (see \cite[\S7.D]{Cox97} and \cite[\S 12]{Neu1992}).
Thereafter, for each $i>1$, the surjection $\Cl(\OO_{i+1}) \rightarrow \Cl(\OO_i)$ has cyclic
kernel of order $\ell$ by virtue of the class number formula (\ref{equation:classnumberformula}),
and hence we have a short exact sequence
\[
1 \rightarrow
\ZZ/\ell\ZZ  \rightarrow
\Cl(\OO_{i+1}) \rightarrow \Cl(\OO_{i}) \rightarrow 1
\]
Thus if we have already constructed some representative for $\psi_A$ modulo $\ell^i\OO_K$,
we can lift it to find $\psi_A\bmod\ell^{i+1}\OO_K$ from $\ell$ possible preimages.
For each candidate lift $\psi_A\bmod\ell^{i+1}\OO_K$, we search for an smooth representative
\[
\psi_A \equiv \psi_1^{e_1}\psi_2^{e_2}\cdot\ldots\cdot\psi_t^{e_t}\bmod\ell^{i+1}\OO_K
\]
with $\deg(\psi_j)=q_j$ small. The candidate smooth lift can be applied to $E_{i+1}$
and the correct lift is that which sends $E_{i+1}$ to $F_{i+1}$ in the $\ell$-isogeny
chain (see Figure~\ref{Figure:KeyReconstruction}).
This yields an algorithm involving multiple instances of the discrete logarithm problem
in a group of order $\ell$ as in Pohlig-Hellman algorithm \cite{PH1978} and in the
generalization of Teske \cite{Tes1999}.

\begin{figure}[ht!]
  \centering
  \includegraphics[height=7cm]{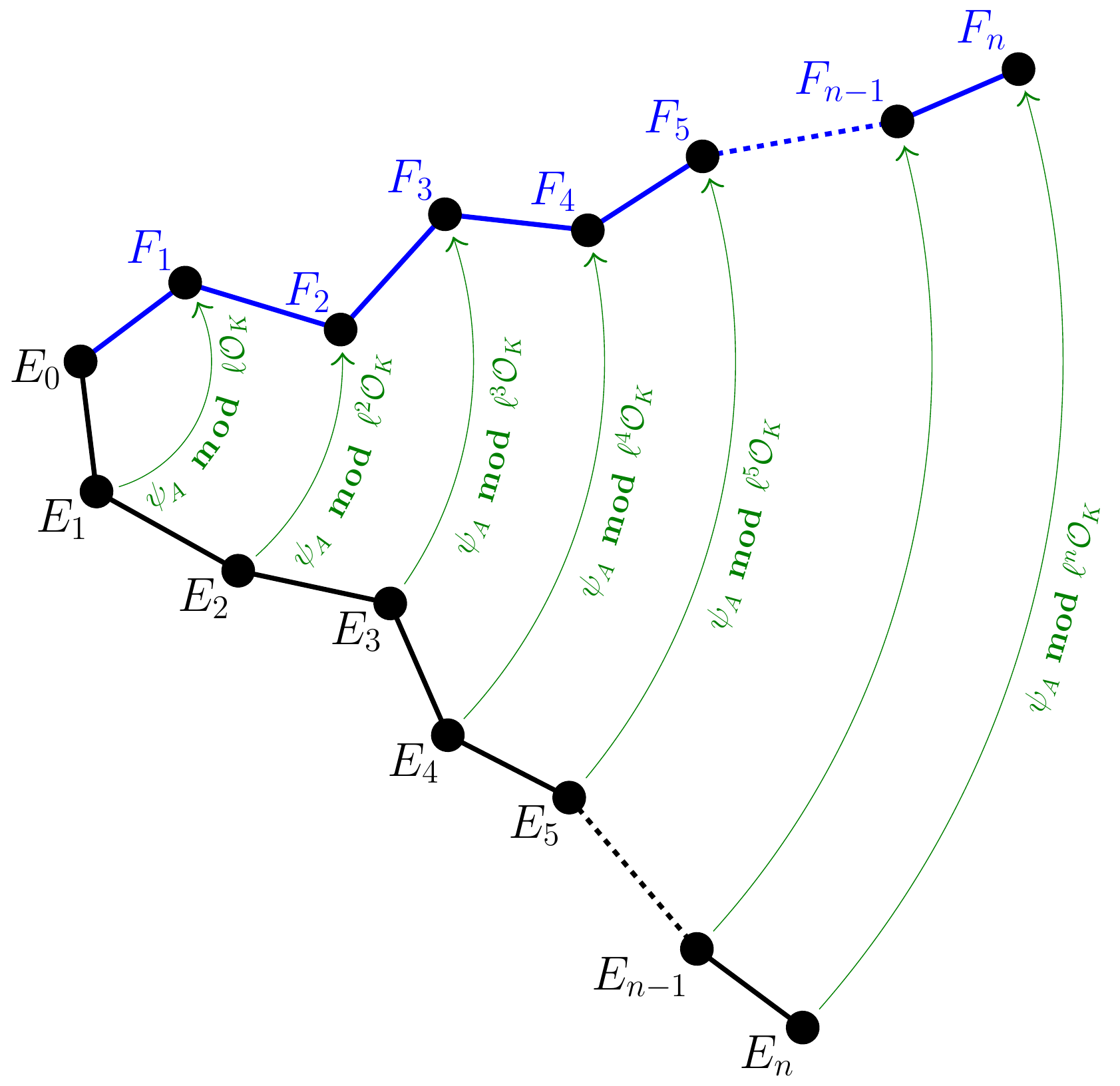}
  \caption{Construction of Alice's secret key}
  \label{Figure:KeyReconstruction}
\end{figure}

In conclusion, this na\"{i}ve protocol is insecure because two parties share the knowledge
of the entire chains $(F_i)$ and $(G_i)$.  The question becomes: how can we avoid sharing
the $\ell$-isogeny chains while still giving the other party enough information to carry
out their isogeny walk?

\subsection{The OSIDH protocol}

We now detail how to send enough public data to compute the isogenies $\psi_A$
and $\psi_B$ on $G = G_n$ and $F = F_n$, respectively, without revealing the
$\ell$-isogeny chains $(F_i)$ and $(G_i)$.  The setup remains the same with a
public choice of $\OO_K$-oriented elliptic curve $E_0$ and $\ell$-isogeny chain
\[
E_0\to E_1\to\cdots\to E_n.
\]
Moreover, a set of primes $\qq_1,\ldots,\qq_t$ (above $q_1,\ldots,q_t$) splitting in $\OO_K$ is fixed.

The first step consists of choosing the secret keys; these are represented
by a sequence of integers $(e_1,\ldots,e_t)$ such that $|e_i|\leq r$.
The bound $r$ is taken so that the number $(2r+1)^t$ of curves that can be reached is
sufficiently large. 
This choice of integers enables Alice to compute a new elliptic curve
\[
F_n=\frac{E_n}{E_n\big[\qq_1^{e_1}\cdots\qq_t^{e_t}\big]}
\]
by means of constructing the following commutative diagram
\begin{center}
	\begin{tikzpicture}
	\draw[fill=black] (0,3) circle [radius=.1]
	(0,2) circle [radius=.1]
	(0,1.1) circle [radius=.01]
	(0,1) circle [radius=.01]
	(0,0.9) circle [radius=.01]
	(0,0) circle [radius=.1];
	\draw[->] (0,2.8) -- (0,2.2);
	\draw[->] (0,1.8) -- (0,1.2);
	\draw[->] (0,0.8) -- (0,0.2);
	\draw[->] (0.2,3) -- (0.8,3);
	\draw[->] (0.2,2) -- (0.8,2);
	\draw[->] (0.2,0) -- (0.8,0);

	\node[anchor=east] at (0,3) {$E_0$};
	\node[anchor=east] at (0,2) {$E_1$};
	\node[anchor=east] at (0,0) {$E_n$};

	\begin{scope}[xshift=1cm]  
	\draw[fill=black] (0,3) circle [radius=.1]
	(0,2) circle [radius=.1]
	(0,1.1) circle [radius=.01]
	(0,1) circle [radius=.01]
	(0,0.9) circle [radius=.01]
	(0,0) circle [radius=.1];
	\draw[->] (0,2.8) -- (0,2.2);
	\draw[->] (0,1.8) -- (0,1.2);
	\draw[->] (0,0.8) -- (0,0.2);
	\draw[->] (0.2,3) -- (0.8,3);
	\draw[->] (0.2,2) -- (0.8,2);
	\draw[->] (0.2,0) -- (0.8,0);

	\node[anchor=south] at (0,3) {$\frac{E_0}{E_0\left[\qq_1\right]}$};
	\node[anchor=south] at (0,3.7) {\rotatebox{90}{$=$}};
	\node[anchor=south] at (0,4) {$E_0$};

	\node[anchor=north] at (0,-0.05) {$F_n^{(1)}$};
	\end{scope}

	\begin{scope}[xshift=2cm]
	\draw[fill=black] (0,3) circle [radius=.01]
	(-0.1,3) circle [radius=.01]
	(0.1,3) circle [radius=.01]
	(0,2) circle [radius=.01]
	(-0.1,2) circle [radius=.01]
	(0.1,2) circle [radius=.01]
	(0,0) circle [radius=.01]
	(-0.1,0) circle [radius=.01]
	(0.1,0) circle [radius=.01];
	\draw[->] (0.2,3) -- (0.8,3);
	\draw[->] (0.2,2) -- (0.8,2);
	\draw[->] (0.2,0) -- (0.8,0);
	\end{scope}

	\begin{scope}[xshift=3cm]  
	\draw[fill=black] (0,3) circle [radius=.1]
	(0,2) circle [radius=.1]
	(0,1.1) circle [radius=.01]
	(0,1) circle [radius=.01]
	(0,0.9) circle [radius=.01]
	(0,0) circle [radius=.1];
	\draw[->] (0,2.8) -- (0,2.2);
	\draw[->] (0,1.8) -- (0,1.2);
	\draw[->] (0,0.8) -- (0,0.2);
	\draw[->] (0.2,3) -- (0.8,3);
	\draw[->] (0.2,2) -- (0.8,2);
	\draw[->] (0.2,0) -- (0.8,0);
	\node[anchor=south] at (-0.1,3) {$\frac{E_0}{E_0\left[\qq_1^{e_1}\right]}$};
	\node[anchor=south] at (-0.1,3.7) {\rotatebox{90}{$=$}};
	\node[anchor=south] at (-0.1,4) {$E_0$};

	\node[anchor=north] at (0,-0.05) {$F_n^{(e_1)}$};
	\end{scope}

	\begin{scope}[xshift=4cm]  
	\draw[fill=black] (0,3) circle [radius=.1]
	(0,2) circle [radius=.1]
	(0,1.1) circle [radius=.01]
	(0,1) circle [radius=.01]
	(0,0.9) circle [radius=.01]
	(0,0) circle [radius=.1];
	\draw[->] (0,2.8) -- (0,2.2);
	\draw[->] (0,1.8) -- (0,1.2);
	\draw[->] (0,0.8) -- (0,0.2);
	\draw[->] (0.2,3) -- (0.8,3);
	\draw[->] (0.2,2) -- (0.8,2);
	\draw[->] (0.2,0) -- (0.8,0);

	\node[anchor=south] at (0.1,3) {$\frac{E_0}{E_0\left[\qq_1^{e_1}\qq_2^1\right]}$};
	\node[anchor=south] at (0.1,3.7) {\rotatebox{90}{$=$}};
	\node[anchor=south] at (0.1,4) {$E_0$};

	\node[anchor=north] at (0,-0.05) {$F_n^{(e_1,1)}$};
	\end{scope}

	\begin{scope}[xshift=5cm]
	\draw[fill=black] (0,3) circle [radius=.01]
	(-0.1,3) circle [radius=.01]
	(0.1,3) circle [radius=.01]
	(0,2) circle [radius=.01]
	(-0.1,2) circle [radius=.01]
	(0.1,2) circle [radius=.01]
	(0,0) circle [radius=.01]
	(-0.1,0) circle [radius=.01]
	(0.1,0) circle [radius=.01];
	\draw[->] (0.2,3) -- (0.8,3);
	\draw[->] (0.2,2) -- (0.8,2);
	\draw[->] (0.2,0) -- (0.8,0);
	\end{scope}

	\begin{scope}[xshift=6cm]  
	\draw[fill=black] (0,3) circle [radius=.1]
	(0,2) circle [radius=.1]
	(0,1.1) circle [radius=.01]
	(0,1) circle [radius=.01]
	(0,0.9) circle [radius=.01]
	(0,0) circle [radius=.1];
	\draw[->] (0,2.8) -- (0,2.2);
	\draw[->] (0,1.8) -- (0,1.2);
	\draw[->] (0,0.8) -- (0,0.2);
	\draw[->] (0.2,3) -- (0.8,3);
	\draw[->] (0.2,2) -- (0.8,2);
	\draw[->] (0.2,0) -- (0.8,0);

	\node[anchor=south] at (0,3) {$\frac{E_0}{E_0\left[\qq_1^{e_1}\qq_2^{e_2}\right]}$};
	\node[anchor=south] at (0,3.7) {\rotatebox{90}{$=$}};
	\node[anchor=south] at (0,4) {$E_0$};

	\node[anchor=north] at (0,-0.05) {$F_n^{(e_1,e_2)}$};
	\end{scope}

	\begin{scope}[xshift=7cm]
	\draw[fill=black] (0,3) circle [radius=.01]
	(-0.1,3) circle [radius=.01]
	(0.1,3) circle [radius=.01]
	(0,2) circle [radius=.01]
	(-0.1,2) circle [radius=.01]
	(0.1,2) circle [radius=.01]
	(0,0) circle [radius=.01]
	(-0.1,0) circle [radius=.01]
	(0.1,0) circle [radius=.01];
	\draw[->] (0.2,3) -- (0.8,3);
	\draw[->] (0.2,2) -- (0.8,2);
	\draw[->] (0.2,0) -- (0.8,0);
	\end{scope}

	\begin{scope}[xshift=8cm]  
	\draw[fill=black] (0,3) circle [radius=.1]
	(0,2) circle [radius=.1]
	(0,1.1) circle [radius=.01]
	(0,1) circle [radius=.01]
	(0,0.9) circle [radius=.01]
	(0,0) circle [radius=.1];
	\draw[->] (0,2.8) -- (0,2.2);
	\draw[->] (0,1.8) -- (0,1.2);
	\draw[->] (0,0.8) -- (0,0.2);
	\draw[->] (0.2,3) -- (0.8,3);
	\draw[->] (0.2,2) -- (0.8,2);
	\draw[->] (0.2,0) -- (0.8,0);

	\node[anchor=south] at (0,3) {$\frac{E_0}{E_0\left[\qq_1^{e_1}\ldots\qq_{t-1}^{e_{t\text{-}1}}\right]}$};
	\node[anchor=south] at (0,3.7) {\rotatebox{90}{$=$}};
	\node[anchor=south] at (0,4) {$E_0$};

	\node[anchor=north] at (0,-0.05) {$F_n^{(e_1,\ldots,e_{t\text{-}1})}$};
	\end{scope}

	\begin{scope}[xshift=9cm]
	\draw[fill=black] (0,3) circle [radius=.01]
	(-0.1,3) circle [radius=.01]
	(0.1,3) circle [radius=.01]
	(0,2) circle [radius=.01]
	(-0.1,2) circle [radius=.01]
	(0.1,2) circle [radius=.01]
	(0,0) circle [radius=.01]
	(-0.1,0) circle [radius=.01]
	(0.1,0) circle [radius=.01];
	\draw[->] (0.2,3) -- (0.8,3);
	\draw[->] (0.2,2) -- (0.8,2);
	\draw[->] (0.2,0) -- (0.8,0);
	\end{scope}

	\begin{scope}[xshift=10cm]  
	\draw[fill=black] (0,3) circle [radius=.1]
	(0,2) circle [radius=.1]
	(0,1.1) circle [radius=.01]
	(0,1) circle [radius=.01]
	(0,0.9) circle [radius=.01]
	(0,0) circle [radius=.1];
	\draw[->] (0,2.8) -- (0,2.2);
	\draw[->] (0,1.8) -- (0,1.2);
	\draw[->] (0,0.8) -- (0,0.2);

	\node[anchor=south] at (0,3) {$\frac{E_0}{E_0\left[\qq_1^{e_1}\ldots\qq_t^{e_t}\right]}$};
	\node[anchor=south] at (0,3.7) {\rotatebox{90}{$=$}};
	\node[anchor=south] at (0,4) {$E_0$};

	\node[anchor=west] at (0,2.8) {$F_0$};
	\node[anchor=west] at (0,2) {$F_1$};
	\node[anchor=west] at (0,0.2) {$F_n$};

	\node[anchor=north] at (0,-0.05) {$F_n^{(e_1,\ldots,e_{t})}$};

	\end{scope}

	\end{tikzpicture}
\end{center}

\noindent {\bf Remark.} Observe that this is just a union of $q_i$-ladders.

At this point the idea is to exchange curves $F_n$ and $G_n$ and
to apply the same process again starting from the elliptic curve
received from the other party. Unfortunately, this is not enough
to get to the same final elliptic curve. Once Alice receives
the unoriented curve $G_n$ computed by Bob she also needs additional
information for each prime $\qq_i$:
\begin{center}
\begin{tikzpicture}
\draw[fill=black] (0,0) circle [radius=.1];
\draw[->](0.2,0) -- (3.2,0);
\draw[->](-0.2,0) -- (-3.2,0);
\node at (0,0.7) {\footnotesize Bob's curve};
\node at (0,0.3) {\footnotesize $G_n$};
\node at (-1.6,-0.2) {\footnotesize Horizontal $p_i$-isogeny};
\node at (-1.6,-0.5) {\footnotesize with kernel $G_n[\bar{\qq}_i]$};
\node at (1.6,-0.2) {\footnotesize Horizontal $p_i$-isogeny};
\node at (1.6,-0.5) {\footnotesize with kernel $G_n[\qq_i]$};
\end{tikzpicture}
\end{center}
but she has no information as to which directions --- out of $q_i+1$
total $q_i$-isogenies --- to take as $\qq_i$ and $\bar{\qq}_i$.
For this reason, once that they have constructed their elliptic curves
$F_n$ and $G_n$, they precompute, for each prime $\qq_i$, the
$q_i$-isogeny chains coming from $\bar{\qq}_i^j$ (denoted by the
class $\qq_i^{-j}$) and $\qq_i^j$:
\[
F_{n,i}^{(-r)}\leftarrow\cdots\leftarrow F_{n,i}^{(-1)}\leftarrow F_n\to F_{n,i}^{(1)}\to\cdots\to F_{n,i}^{(r-1)}\to F_{n,i}^{(r)}
\]
and
\[
G_{n,i}^{(-r)}\leftarrow\cdots\leftarrow G_{n,i}^{(-1)}\leftarrow G_n\to G_{n,i}^{(1)}\to\cdots\to G_{n,i}^{(r-1)}\to G_{n,i}^{(r)}
\]
Now Alice obtains from Bob the curve $G_n$ and, for each $i$, the horizontal
$q_i$-isogeny chains determined by the isogenies with kernels $G_n[\qq_i^j]$.
With this information Alice can take $e_1$ steps in the $\qq_1$-isogeny chain and
push forward all the $\qq_i$-isogeny chains for $i>1$.

\noindent {\bf Remark.} We recall that pushing forward means constructing a ladder which transmits all the information about the commutative action of $\qq_i^{e_i}$ in the class group.

\begin{center}
\begin{tikzpicture}

\node at (0,0) {\footnotesize $G_n$};
\draw[->] (-0.3,0)--(-1.05,0);
\draw[->] (0.3,0)--(1.05,0);
\draw[->,green] (0,0.2)--(0,1.2);
\draw[->,green] (0,-0.2)--(0,-1.2);
\draw[->,red] (0.1,0.1*\sqt)--(0.5,0.5*\sqt); \draw[->,red] (-0.1,-0.1*\sqt)--(-0.5,-0.5*\sqt);
\draw[->,blue] (0.1*\sqt,0.1)--(0.5*\sqt,0.5); \draw[->,blue] (-0.1*\sqt,-0.1)--(-0.5*\sqt,-0.5);
\node at (0.75,-0.2) {\footnotesize $\textcolor{black}{\qq_1}$};
\node at (-0.2,0.75) {\footnotesize $\textcolor{green}{\qq_2}$};
\node at (0.6,0.6*\sqt) {\footnotesize $\textcolor{red}{\qq_3}$};
\node at (0.6*\sqt,0.6) {\footnotesize $\textcolor{blue}{\qq_4}$};
\node at (-1.5,0) {\footnotesize $G_{n,1}^{(-1)}$};
\node at (1.5,0) {\footnotesize $G_{n,1}^{(1)}$};
\node at (0,1.5) {\footnotesize $G_{n,2}^{(1)}$};
\node at (0,-1.5) {\footnotesize $G_{n,2}^{(-1)}$};
\node at (3,0) {\footnotesize $G_{n,1}^{(2)}$};
\node at (5,0) {\footnotesize $G_{n,1}^{(r)}$};
\node at (-3,0) {\footnotesize $G_{n,1}^{(-2)}$};
\node at (-5,0) {\footnotesize $G_{n,1}^{(e_1)}$};
\node at (-7,0) {\footnotesize $G_{n,1}^{(-r)}$};
\draw[fill=black] (-4,0) circle [radius=0.01]
(-4.1,0) circle [radius=0.01]
(-3.9,0) circle [radius=0.01]
(-6,0) circle [radius=0.01]
(-6.1,0) circle [radius=0.01]
(-5.9,0) circle [radius=0.01];

\draw[fill=black] (4,0) circle [radius=0.01]
(4.1,0) circle [radius=0.01]
(3.9,0) circle [radius=0.01];

\node at (0,5.5) {\footnotesize $G_{n,2}^{(r)}$};
\node at (0,-3.5) {\footnotesize $G_{n,2}^{(-r)}$};

\draw[->] (-1.95,0)--(-2.55,0);
\draw[->] (-3.45,0)--(-3.8,0);
\draw[->] (-4.2,0)--(-4.55,0);
\draw[->] (-5.45,0)--(-5.8,0);
\draw[->] (-6.2,0)--(-6.55,0);
\draw[->] (1.95,0)--(2.55,0);
\draw[->] (3.45,0)--(3.8,0);
\draw[->] (4.2,0)--(4.55,0);
\draw[black,fill=black] (0,2.4) circle [radius=0.01]
(0,2.5) circle [radius=0.01]
(0,2.6) circle [radius=0.01];
\draw[->,black] (0,1.8)--(0,2.3);
\draw[->,black] (0,2.7)--(0,3.2);
\draw[black,fill=black] (0,4.4) circle [radius=0.01]
(0,4.5) circle [radius=0.01]
(0,4.6) circle [radius=0.01];
\draw[->,black] (0,3.8)--(0,4.3);
\draw[->,black] (0,4.7)--(0,5.2);
\draw[black,fill=black] (0,-2.4) circle [radius=0.01]
(0,-2.5) circle [radius=0.01]
(0,-2.6) circle [radius=0.01];
\draw[->,black] (0,-1.8)--(0,-2.3);
\draw[->,black] (0,-2.7)--(0,-3.2);
\draw[fill=black] (-1.5,1.5) circle [radius=.1];
\draw[->] (-0.4,1.5)--(-1.3,1.5);
\draw[->,black] (-1.5,0.3)--(-1.5,1.3);
\draw[fill=black] (-3,1.5) circle [radius=.1];
\draw[->] (-1.7,1.5)--(-2.8,1.5);
\draw[->] (-3.2,1.5)--(-3.8,1.5);
\draw[fill=black] (-4,1.5) circle [radius=0.01]
(-4.1,1.5) circle [radius=0.01]
(-3.9,1.5) circle [radius=0.01];
\draw[->] (-4.2,1.5)--(-4.55,1.5);
\draw[->,black] (-3,0.3)--(-3,1.3);
\draw[->,black] (-5,0.3)--(-5,1.2);
\node at (-5,1.5) {\footnotesize $G_{n,2}^{(e_1,1)}$};
\draw[fill=black] (-1.5,3.5) circle [radius=.1];
\draw[->] (-0.4,3.5)--(-1.3,3.5);
\draw[->,black] (-1.5,1.7)--(-1.5,2.3);
\draw[black,fill=black] (-1.5,2.4) circle [radius=0.01]
(-1.5,2.5) circle [radius=0.01]
(-1.5,2.6) circle [radius=0.01];
\draw[->,black] (-1.5,2.7)--(-1.5,3.3);
\draw[fill=black] (-3,3.5) circle [radius=.1];
\draw[->] (-1.7,3.5)--(-2.8,3.5);
\draw[->] (-3.2,3.5)--(-3.8,3.5);
\draw[fill=black] (-4,3.5) circle [radius=0.01]
(-4.1,3.5) circle [radius=0.01]
(-3.9,3.5) circle [radius=0.01];
\draw[->] (-4.2,3.5)--(-4.5,3.5);
\draw[->,black] (-3,1.7)--(-3,2.3);
\draw[black,fill=black] (-3,2.4) circle [radius=0.01]
(-3,2.5) circle [radius=0.01]
(-3,2.6) circle [radius=0.01];
\draw[->,black] (-3,2.7)--(-3,3.3);
\draw[black,fill=black] (-5,2.4) circle [radius=0.01]
(-5,2.5) circle [radius=0.01]
(-5,2.6) circle [radius=0.01];
\draw[->,black] (-5,2.7)--(-5,3.2);
\draw[->,black] (-5,1.8)--(-5,2.3);
\draw[->,black,dashed] (-5,2.7)--(-5,3.2);
\draw[->,black,dashed] (-5,1.8)--(-5,2.3);
\node at (-5,3.5) {\footnotesize $\textcolor{black}{G_{n,2}^{(e_1,e_2)}}$};
\draw[->,red] (-5+0.15,0.15*\sqt)--(-5+0.5,0.5*\sqt); \draw[->,red] (-5-0.15,-0.15*\sqt)--(-5-0.5,-0.5*\sqt);
\draw[->,blue] (-5+0.2*\sqt,0.2)--(-5+0.5*\sqt,0.5); \draw[->,blue] (-5-0.15*\sqt,-0.15)--(-5-0.5*\sqt,-0.5);
\draw[->,red] (-5+0.15,0.15*\sqt+3.5)--(-5+0.5,0.5*\sqt+3.5); \draw[->,red] (-5-0.15,-0.15*\sqt+3.5)--(-5-0.5,-0.5*\sqt+3.5);
\draw[->,blue] (-5+0.22*\sqt,0.22+3.5)--(-5+0.5*\sqt,0.5+3.5); \draw[->,blue] (-5-0.2*\sqt,-0.2+3.5)--(-5-0.5*\sqt,-0.5+3.5);

\node at (0,3.5) {\footnotesize $G_{n,2}^{(e_2)}$};
\end{tikzpicture}
\end{center}
Alice repeats the process for all the $\qq_i$'s every time pushing forward the
isogenies for the primes with index strictly bigger than $i$. Finally, she
obtains a new elliptic curve
\[
H_n=\frac{E_n}{E_n\big[\qq_1^{e_1+d_1}\cdots\qq_t^{e_t+d_t}\big]}
\]
Bob follows the same process with the public data received from Alice, in order
to compute the same curve $H_n$. Recall that, in the naive protocol, Alice and Bob compute
the group action on the full $\ell$-isogeny chains:
\begin{center}
\begin{tikzpicture}
\node at (0,4) {$E_0$}; \node at (1,4) {$E_1$}; \node at (2,4) {$E_2$}; \node at (4,4) {$E_n$};
\draw[->] (0.3,4) -- (0.7,4);
\draw[->] (1.3,4) -- (1.7,4);
\draw[->] (2.3,4) -- (2.8,4);
\draw (3,4) circle [radius=.01]
(2.9,4) circle [radius=.01]
(3.1,4) circle [radius=.01];
\draw[->] (3.2,4) -- (3.7,4);

\node at (7,4) {$E_0$};\node at (8,4) {$G_1$}; \node at (9,4) {$G_2$}; \node at (11,4) {$G_n$};
\draw[->] (7.3,4) -- (7.7,4);
\draw[->] (8.3,4) -- (8.7,4);
\draw[->] (9.3,4) -- (9.8,4);
\draw
 (10,4) circle [radius=.01]
 (9.9,4) circle [radius=.01]
 (10.1,4) circle [radius=.01];
\draw[->] (10.2,4) -- (10.7,4);

\node at (0,2) {$E_0$}; \node at (1,2) {$F_1$}; \node at (2,2) {$F_2$}; \node at (4,2) {$F_n$};
\draw[->] (0.3,2) -- (0.7,2);
\draw[->] (1.3,2) -- (1.7,2);
\draw[->] (2.3,2) -- (2.8,2);
\draw (3,2) circle [radius=.01]
(2.9,2) circle [radius=.01]
(3.1,2) circle [radius=.01];
\draw[->] (3.2,2) -- (3.7,2);

\node at (7,2) {$E_0$};\node at (8,2) {$H_1$}; \node at (9,2) {$H_2$}; \node at (11,2) {$H_n$};
\draw[->] (7.3,2) -- (7.7,2);
\draw[->] (8.3,2) -- (8.7,2);
\draw[->] (9.3,2) -- (9.8,2);
\draw
  (10,2) circle [radius=.01]
  (9.9,2) circle [radius=.01]
  (10.1,2) circle [radius=.01];
\draw[->] (10.2,2) -- (10.7,2);

\def\myshift#1{\raisebox{+0.5ex}}
\def\myshiftbis#1{\raisebox{-2ex}}
\draw[blue,->,postaction={decorate,decoration={text along path,text align=center,text color=blue,text={|\scriptsize\sffamily\myshift|Alice}}}] (2,3.7) -- (2,2.3);
\draw[blue,->,postaction={decorate,decoration={text along path,text align=center,text color=blue,text={|\scriptsize\sffamily\myshift|Alice}}}] (9,3.7) -- (9,2.3);
\draw[red,->,postaction={decorate,decoration={text along path,text align=center,text color=red,text={|\scriptsize\sffamily\myshift|Bob}}}] (4.3,4) -- (6.7,4);
\draw[red,->,postaction={decorate,decoration={text along path,text align=center,text color=red,text={|\scriptsize\sffamily\myshift|Bob}}}] (4.3,2) -- (6.7,2);

\end{tikzpicture}
\end{center}
In the refined OSIDH protocol, Alice and Bob share sufficient information to determine
the curve $H_n$ without knowledge of the other party's $\ell$-isogeny chain $(G_i)$
and $(F_i)$, nor the full $\ell$-isogeny chain $(H_i)$ from the base curve $E_0$.

\begin{center}
\begin{tabular}{L{3cm} C{4cm} C{4.1cm}}
\hline
\multicolumn{3}{L{11.1cm}}{{\bf PUBLIC DATA:}
A descending $\ell$-isogeny chain $E_0\to E_1\to\cdots\to E_n$ and a set of splitting primes
$\qq_1,\ldots,\qq_t\subseteq\OO = \End(E_n)\cap K\hookrightarrow\OO_K$}\\
\hline
&{\bf ALICE}&{\bf BOB}\\
Choose integers in an interval $[-r,r]$&${\displaystyle(e_1,\ldots,e_t)}$&${\displaystyle(d_1,\ldots,d_t)}$\\
Construct an isogenous curve&
  ${\displaystyle F_n=\frac{E_n}{E_n\big[\qq_1^{e_1}\cdots\qq_t^{e_t}\big]}}$&
  ${\displaystyle G_n=\frac{E_n}{E_n\big[\qq_1^{d_1}\cdots\qq_t^{d_t}\big]}}$\\
Precompute all directions $\forall~i$&
  ${F_n\to F_{n,i}^{(1)}\to\cdots\to F_{n,i}^{(r)}}$&
  ${G_n\to G_{n,i}^{(1)}\to\cdots\to G_{n,i}^{(r)}}$\\
... and their conjugates&
  $\underbrace{F_{n,i}^{(-r)}\!\!\leftarrow\cdots\leftarrow F_{n,i}^{(-1)}\!\!\leftarrow F_n}_{\tikzmark{alpha}}$ &
  $\underbrace{G_{n,i}^{(-r)}\!\!\leftarrow\cdots\leftarrow G_{n,i}^{(-1)}\!\!\leftarrow G_n}_{\tikzmark{beta}}$\\
Exchange data&&\\
&$G_n+$directions\tikzmark{gamma}&\tikzmark{delta}$F_n+$directions
\begin{tikzpicture}
[
remember picture,
overlay,
-latex,
color=green,
yshift=1ex,
shorten >=1pt,
shorten <=1pt,
]
\draw ({pic cs:alpha}) -- ({pic cs:delta});
\draw ({pic cs:beta}) -- ({pic cs:gamma});
\end{tikzpicture}\\
Compute shared data&
  Takes $e_i$ steps in $\qq_i$-isogeny chain \& push forward information for~all~$j>i$.&
  Takes $d_i$ steps in $\qq_i$-isogeny chain \& push forward information for~all~$j>i$.\\
\hline
\multicolumn{3}{L{11.1cm}}{In the end, Alice and Bob share the same elliptic curve} \\
\multicolumn{3}{C{11.1cm}}{
$
H_n \displaystyle
 = \frac{F_n}{F_n\!\big[\qq_1^{d_1}\cdots\qq_t^{d_t}\big]}
 = \frac{G_n}{G_n\!\big[\qq_1^{e_1}\cdots\qq_t^{e_t}\big]}
 = \frac{E_n}{E_n\!\big[\qq_1^{e_1+d_1}\cdots\qq_t^{e_t+d_t}\big]}\cdot$}\\[4mm]
\hline
\end{tabular}
\end{center}
\begin{figure}[ht!]\label{Figure:OSIDH}
\includegraphics[width=\textwidth]{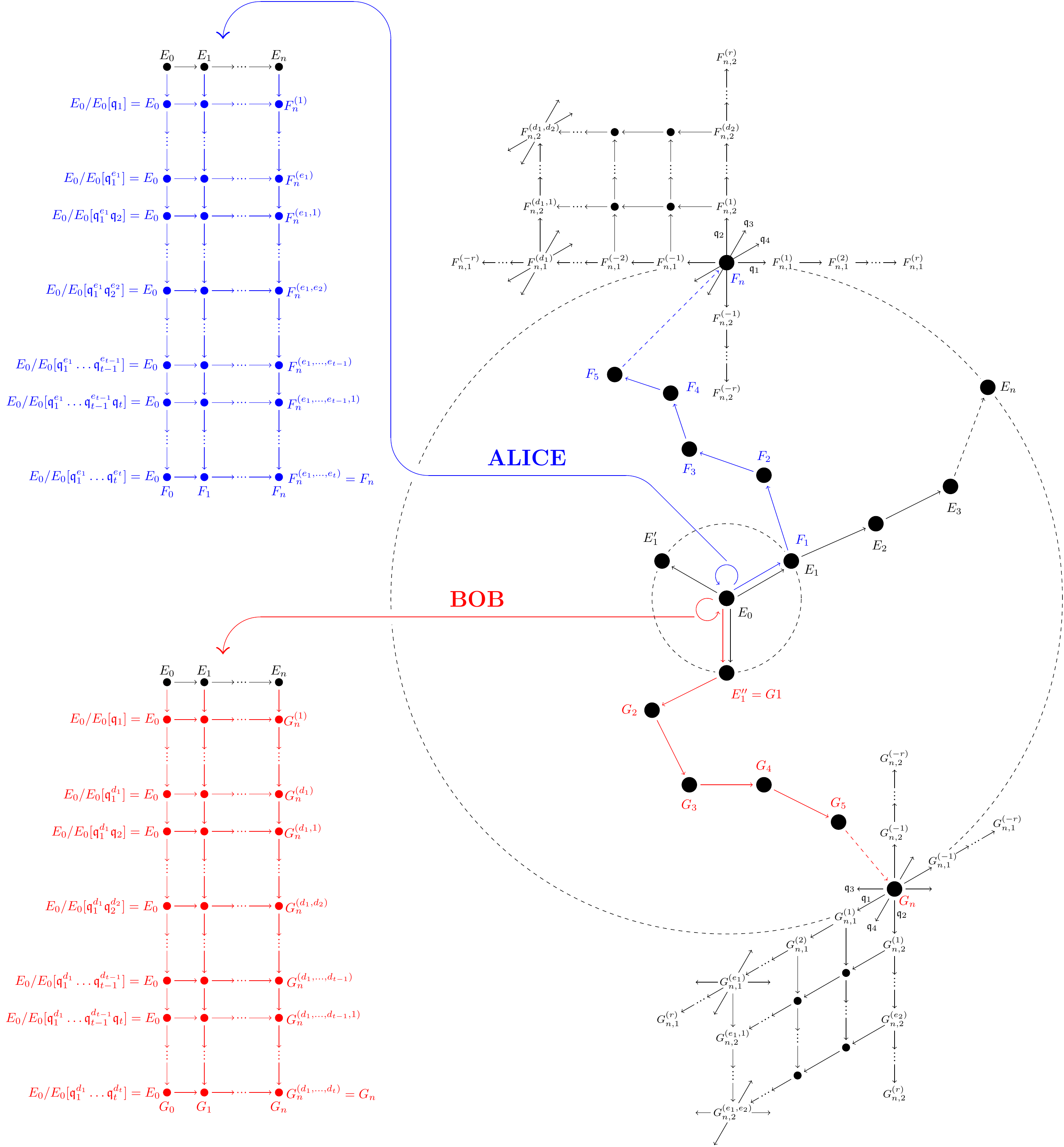}
\caption{Graphic representation of OSIDH}
\end{figure}

\noindent {\bf Remark.}
We can read this scheme using the terminology of section \ref{section:ActionClassGroup}.

After the choice of the secret key, we observe a vortex: Alice (respectively Bob) acts on an isogeny crater (that in the case of $\OO_K=\ZZ\left[\zeta_3\right]$ or $\ZZ\left[i\right]$ consists of a single points) with the primes $\qq_1^{e_1}\cdot\ldots\cdot\qq_t^{e_t}$ (respectively $\qq_1^{d_1}\cdot\ldots\cdot\qq_t^{d_t}$).

This action is eventually transmitted along the $\ell$-isogeny chain and we get a whirlpool. We can think of the isogeny volcano as rotating under the action of the secret keys and the initial $\ell$-isogeny path transforming into the two secret isogeny chains.

\section{Security considerations}\label{section:Securityparameters}

In order to ensure security of the system, we have seen that the data giving
the orientation must remain hidden.  A second consideration is the proportion
of curves attained by the action of the class group $\Cl(\OO)$, and by the
private walks $\psi_A$ and $\psi_B$ of Alice and Bob in that class group.
The size of the orbit of $\Cl(\OO)$ is controlled by the chain length $n$,
and the number of curves attained by the private walks is further limited
by the prime power data, up to exponent bounds, which we allow ourselves to
transmit.

\subsection*{Chain length}

Suppose that $(E_i)$ is an isogeny chain of length $n$, from a supersingular
elliptic curve $E_0$ oriented by $\OO_K$ of class number one, and consider
$$
\Hom(E_0,E_n) = \phi\OO_K + \psi\OO_K.
$$
As a quadratic module with respect to the degree map, its determinant is $p^2$.
If the length $n$ is of sufficient length such that $E_n$ represents a general
curve in $\mathrm{SS}(p)$, then a set of reduced basis elements $\phi$ and
$\psi$ satisfies
$$
\deg(\phi) \approx \deg(\psi) \approx \sqrt{p}.
$$

Now suppose that $\phi:E_0 \rightarrow E_n$ is the isogeny giving the
$\ell$-isogeny chain.  If $\deg(\phi) = \ell^n$ is less than $\sqrt{p}$,
then $\phi\OO_K$ is a submodule generated by short isogenies, and
$E_n$ is special.  We conclude that we must choose $n$ to be at least
$\log_\ell(p)/2$ in order to avoid an attack which seeks to determine
$\phi\OO_K$ as a distinguished submodule of low degree isogenies.

We extend this argument to consider the logarithmic proportion $\lambda$ of
supersingular elliptic curves we can reach.  In order to cover $p^\lambda$
supersingular curves, out of $|\mathrm{SS}(p)| = p/12 + \varepsilon_p$ curves,
$\deg(\phi)$ must be such that
$$
|\Cl(\OO)| = \left|\frac{\left(\OO_K/\ell^n\OO_K\right)^*}{\OO_K^*(\ZZ/\ell^n\ZZ)^*}\right|
\approx \ell^n = \deg(\phi) \approx p^\lambda.
$$
In particular, choosing $\lambda = 1$, we find that $n = \log_\ell(p)$ is the
critical length for reaching all supersingular curves.

\subsection*{Degree of private walks}

Suppose now that $E = E_n$ is a generic supersingular curve and $F$ another. Without
an $\OO_K$-module structure, we have a basis $\{\psi_1,\psi_2,\psi_3,\psi_4\}$ such
that
$$
\Hom(E,F) = \ZZ\psi_1 + \ZZ\psi_2 + \ZZ\psi_3 + \ZZ\psi_4.
$$
Assuming that $E$ and $F$ are generic relative to one another, a reduced basis
satisfies $\deg(\psi_i) \approx \sqrt{p}$, as above.  Thus the private walk
$\psi_A$ should satisfy
$$
\log_p(\deg(\psi_A)) \ge \frac{1}{2}
$$
in order that $\ZZ\psi_A$ is not a distinguished submodule of $\Hom(E,F)$.
This critical distance is the maximal that can be attained by the SIDH
protocol.

As above, another measure of the generality of $\psi_A$ is the number of curves
that can be reached by different choices of the isogeny $\psi_A$.  For a fixed
degree $m$, the number of curves which can be attained is
$$
|\PP(E[m])| \isom |\PP^1(\ZZ/m\ZZ)| \approx m.
$$
For the SIDH protocol, on has $\ell_A^{n_A} \approx \ell_B^{n_B} \approx
\sqrt{p}$, and only $\sqrt{p}$ curves out of $p/12$ can be reached.

In the CSIDH or OSIDH protocols, the degree of the isogeny is not fixed. The
total number of isogenies of any degree $d$ up to $m$ is
$$
\sum_{d=1}^m |\PP(E[d])| \approx m^2,
$$
but the choice of $\psi_A$ is restricted to a subset of $\OO$-oriented isogenies
in $\Cl(\OO)$.  Such isogenies are restricted to a class proportional to $m$.
Specifically, in the OSIDH construction, if we let $S_m \subset \OO_K$ be the
set of endomorphisms of degree up to $m$, and consider the map
$$
S_m \subset \OO_K \longrightarrow
  \frac{(\OO_K/\ell^n\OO_K)^*}{\OO_K^*(\ZZ/\ell^n\ZZ)^*} \isom \Cl(\OO).
$$
Since $|S_m| \approx m$, to cover a subset of $p^\lambda$ classes, we need
$\log_p(\deg(\psi_A)) \ge \lambda$.

\subsection*{Private walk exponents}

In practice, rather than bounding the degree, for efficient evaluation one fixes
a subset of small split primes, and the space of exponent vectors is bounded.
The instantiation CSIDH-512 (see~\cite{CLMPR2018}) uses a prime of 512 bits such
that for each of 74 primes one has a choice of 11 exponents in $[-5,5]$.
This gives 256 bits of
freedom which is of the order of magnitude to cover $h(-p) \approx \sqrt{p}$
classes (up to logarithmic factors). In this instance the class number $h(-p)$
was computed~\cite{CSI-FISH} and found to be 252 bits.

For the general OSIDH construction, we choose exponent vectors $(e_1,\dots,e_t)$ in
the space $I_1 \times \cdots \times I_t \subset \ZZ^t$, where $I_j = [-r_j,r_j]$,
defining $\psi_A$ with kernel
$$
\ker(\psi_A) = E[\qq_1^{e_1} \cdots \qq_t^{e_t}].
$$
We thus express the map to $\mathrm{SS}(p)$ as the composite of the map of
exponent vectors to the class group and the image of $\Cl(\OO)$:
$$
\prod_{j=1}^t I_j \longrightarrow \Cl(\OO) \longrightarrow \mathrm{SS}(p).
$$
In order to avoid revealing any cycles, we want the former map to be effectively
injective --- either injective or computationally difficult to find a nontrivial
element of the kernel in
$$
(I_1 \times \cdots \times I_t) \cap \ker(\ZZ^t\rightarrow\Cl(\OO)).
$$
In order to cover as many classes as possible, the latter should be nearly
surjective.  Supposing that the former map is injective with image of size
$p^\lambda$ in $\mathrm{SS}(\OO)$, this gives
$p^\lambda < \prod_{j=1}^t (2r_j+1) < |\Cl(\OO)| \approx \ell^n$.
For fixed $r = r_j$, this gives
$$
n > t\log_\ell(2r+1) > \lambda\log_\ell(p).
$$
Setting $\lambda = 1$, $\ell = 2$ and $\log_\ell(p) = 256$, the parameters $t = 74$
and $r = 5$ give critical values as in CSIDH-512, with group action mapping to
the full set of supersingular points $\mathrm{SS}(p)$.

\section{Conclusion}

By imposing the data of an orientation by an imaginary quadratic ring $\OO$,
we obtain an augmented category of supersingular curves on which the class
group $\Cl(\OO)$ acts faithfully and transitively.  This idea is already
implicit in the CSIDH protocol, in which supersingular curves over $\FF_p$
are oriented by the Frobenius subring $\ZZ[\pi] \isom \ZZ[\sqrt{-p}]$.
In contrast we consider an elliptic curve $E_0$ oriented by a CM order
$\OO_K$ of class number one.  To obtain a nontrivial group action, we consider
descending $\ell$-isogeny chains in the $\ell$-volcano, on which the class
group of an order $\OO$ of large index $\ell^n$ in $\OO_K$ acts.
The map from an $\ell$-isogeny chain to its terminal node forgets the structure
of the orientation, giving rise to a generic curve in the supersingular
isogeny graph.  Within this general framework we define a new oriented
supersingular isogeny Diffie-Hellman (OSIDH) protocol,
which has fewer restrictions on the proportion of supersingular curves
covered and on the torsion group structure of the underlying curves.
Moreover, the group action can be carried out effectively solely on the
sequences of modular points (such as $j$-invariants) on a modular curve,
thereby avoiding expensive isogeny computations, and is further amenable
to speedup by precomputations of endomorphisms on the base curve $E_0$.

\newcommand{\LNCS}{Lecture Notes in Computer Science}

\end{document}